\documentclass[12pt]{amsart}

\usepackage{amscd,amssymb,amsthm,verbatim}
\usepackage{enumerate}
\usepackage{bm}
\usepackage{mathrsfs, graphicx} 
\usepackage{stmaryrd}

\usepackage[%
bookmarks=true,
colorlinks,
linkcolor=blue,
urlcolor=blue,
citecolor=blue,
plainpages=false,
pdfpagelabels,
final,
breaklinks=true\between
]{hyperref}
\usepackage{hyperref}
\usepackage[numbers,sort&compress]{natbib}

\numberwithin{equation}{section}

\theoremstyle{plain}
\newtheorem{thm}{Theorem}[section]
\newtheorem{prop}[thm]{Proposition}
\newtheorem{lem}[thm]{Lemma}
\newtheorem{cor}[thm]{Corollary}

\theoremstyle{definition}
\newtheorem{defn}[thm]{Definition}
\theoremstyle{remark}
\newtheorem{rk}[thm]{Remark}
\newcommand{\ve}{\varepsilon}

\newcommand{\mc}[1]{\mathcal{#1}}

\newcommand{\Diff}{\operatorname{Diff}}

\newcommand{\tr}{\operatorname{tr}}
\newcommand{\grd}{\operatorname{grad}}
\newcommand{\Ric}{\mathrm{Ric}}

\newcommand{\dist}{\mathrm{dist}}

\newcommand{\loc}{\mathrm{loc}}

\newcommand{\Rm}{\mathrm{Rm}}

\newcommand{\Div}{\operatorname{div}}

\newcommand{\csum}{\mathbin{\#}}

\newcommand{\current}[1]{\llbracket #1\rrbracket}

\DeclareFontFamily{U}{MnSymbolC}{}
\DeclareSymbolFont{MnSyC}{U}{MnSymbolC}{m}{n}
\DeclareFontShape{U}{MnSymbolC}{m}{n}{
	<-6>  MnSymbolC5
	<6-7>  MnSymbolC6
	<7-8>  MnSymbolC7
	<8-9>  MnSymbolC8
	<9-10> MnSymbolC9
	<10-12> MnSymbolC10
	<12->   MnSymbolC12}{}
\DeclareMathSymbol{\intprod}{\mathbin}{MnSyC}{'270}

\begin{document}
\title[The Positive Mass Theorem with Arbitrary Ends]
{The Positive Mass Theorem with Arbitrary Ends}
\author{Martin Lesourd, Ryan Unger, and Shing-Tung Yau}

\address{Black Hole Initiative, Harvard University, Cambridge, MA 02138}
\email{mlesourd@fas.harvard.edu}

\address{Department of Mathematics, Princeton University, Princeton, NJ 08544}
\email{runger@math.princeton.edu}

\address{Department of Mathematics, Harvard University, Cambridge, MA 02138}
\email{yau@math.harvard.edu}
\maketitle

\begin{abstract}
We prove a Riemannian positive mass theorem for manifolds with a single asymptotically flat end, but otherwise arbitrary other ends, which can be incomplete and contain negative scalar curvature. The incompleteness and negativity is compensated for by large positive scalar curvature on an annulus, in a quantitative fashion. In the complete noncompact case with nonnegative scalar curvature, we have no extra assumption and hence prove a long-standing conjecture of Schoen and Yau. 

\end{abstract}

\tableofcontents

\section{Introduction}
The Riemannian positive mass theorem is a landmark result in geometric analysis and the study of scalar curvature.

\begin{thm}[Schoen--Yau \cite{SY79PMT}]\label{pmt}
    Let $(M^3,g)$ be a complete asymptotically Schwarzschild manifold with finitely many ends which are all asymptotically Schwarzschild. Then the mass of each end is nonnegative.
\end{thm}

The positive mass theorem also involves a rigidity statement when the mass is zero, but we will not concern ourselves with rigidity questions in this paper. This theorem was generalized to the (now commonplace) general asymptotically flat setting in \cite{SY81ii} via the \emph{density theorem}, to dimensions $\le 7$ by Schoen in \cite{SchoenPMT} by studying the \emph{strong stability inequality}, and finally to all dimensions in \cite{SY17} using $k$-slicings and the convenient compactification technique of Lohkamp \cite{L99}.\\ \indent 
This theorem was furthermore generalized to initial data sets in \cite{SY81i} and to dimensions $\le 7$ by Eichmair--Huang--Lee--Schoen \cite{EHLS}. There are numerous other contributions to this field, particularly the case of spin manifolds, which was handled by Witten \cite{Witten}. We note that an alternative approach to the higher dimensional case has been proposed by Lohkamp \cite{L16, L17}.

In this paper, we prove a generalization of this theorem to the setting of manifolds with a distinguished asymptotically flat end and otherwise arbitrary and possibly incomplete other ends. For precision, see Definition \ref{maindef}.

\begin{thm}\label{thm1}
Let $(M^n,g)$, $3\le n\le 7$, be a complete non-compact manifold with nonnegative scalar curvature and suppose that $(M^n,g)$ has at least one asymptotically Schwarzschild end. Then the mass of the distinguished end is nonnegative. 
\end{thm}

This theorem was long conjectured by Schoen and the third author \cite{SY88, SYbook, LUY20}, where it was shown to imply \cite{SY88} the following \emph{Liouville theorem}.
\begin{thm}\label{Liouville}
    Let $(M^n,g)$, $n\geq 3$, be a complete, locally conformally flat manifold with nonnegative scalar curvature. If $\Phi:M^n\to S^n$ a conformal map, then $\Phi$ is injective and $\partial\Phi(M)$ has zero Newtonian capacity.
\end{thm} 

Recently, the present authors \cite{LUY20} carried out Lohkamp's idea in this setting, making use of good estimates on the Green's function to simplify the asymptotics, to show that Theorem \ref{Liouville} actually follows from the impossibility of $R>0$ on $T^n\csum X^n$ for arbitrary $X^n$. The present paper gives a new proof completely in the spirit of what is outlined in \cite{SY88}.

\begin{rk}
Manifolds with complicated ends satisfying the hypotheses of Theorem \ref{thm1} can be constructed using connected sums \cite{GLcsum}. 
\end{rk}

The difficulty in proving Theorem \ref{thm1} is that some of the standard approaches to Theorem \ref{pmt} do not seem to apply. The usual procedure of using density theorems to obtain harmonic asymptotics is hampered by the lack of control of the PDEs on the arbitrary non-compact ends. That being said, it might be possible to produce harmonic asymptotics on the asymptotically flat end via a similar argument to the density theorem. That this can be done in a complete way is not at all clear and it seems like the geometry on the arbitrary ends can be altered significantly in this process. So this approach might work for Theorem \ref{thm1}, but it is doubtful that it can yield our quantitative Theorem \ref{thm2}. 

The reduction of Theorem \ref{thm1} to a geometric statement about $R>0$ via a Lokhamp-type compactification seems to be unavailable for the same reason. This is somewhat tantalizing because the analog of the geometric statement that proves Theorem \ref{pmt} is known: $T^n\#X^n$ admits no complete metric with $R>0$ when $X^n$ is arbitrary by \cite{CL20}; cf.\ \cite{LUY20} for a similar result that applies to a topologically wider class of spaces under an additional geometric assumption for $4\leq n \leq 7$. Finally, spinor based proofs also fall short because they involve controlling boundary integrals in the arbitrary ends, about which nothing is assumed. 

In view of these obstacles, the only viable approach seems to be the original proof \cite{SY79PMT, SchoenPMT}, which constructs a limiting hypersurface from solutions to the Plateau problem in the asymptotically flat end on which the mass is assumed negative. This method cannot be naively implemented however, because nothing seems to stop the solutions from venturing out into the arbitrary ends, which would ruin the asymptotic flatness and completeness of the limiting hypersurface and conformal metric that is crucial for the rest of the argument.   

\begin{rk}\label{Rick}
When $n=3$, Richard Schoen has pointed out to us that Theorem \ref{thm1} follows from essentially the original argument in \cite{SY79PMT} as follows:  The complete minimal surface constructed must be conformally equivalent to $\Bbb C$ by \cite{FCS80, SY82} and the fact that the scalar curvature can be taken strictly positive on the asymptotically Schwarzschild end (see Section \ref{sec:thm1} below). It follows that since it is asymptotically planar, it can only enter the core a bounded amount. At this point, the rest of the argument goes through unchanged. This does not give our stronger Theorem \ref{thm2} below, however. 
\end{rk}

To overcome the issues in dimensions $4\le n\le 7$, we minimize a version of Gromov's $\mu$-bubble functional \cite{G18, GromovFour}, which we construct to prevent solutions from exploring the other ends. The $\mu$-bubble functional, which has recently found great success in noncompact scalar curvature problems \cite{CL20, G20, Richard, Z20, Z21}, is defined on Caccioppoli sets $\Omega$ by 
\begin{equation}
\mathcal A(\Omega)=\mathcal H^{n-1}(\partial^*\Omega)-\int_M(\chi_\Omega -\chi_{\Omega_0}) h\,d\mathcal H^n,
\end{equation}
where $h$ is a function tending to $\pm\infty$ at finite distance from the region of interest and $\Omega_0$ is a reference set. This allows for a trivial construction of barrier surfaces and stable $\mu$-bubbles (which have prescribed mean curvature $h$) will be Yamabe positive as long as the following condition is satisfied:
\begin{equation}\label{mu}R_g+h^2-2|\nabla_gh|>0.\end{equation}
So to use $\mu$-bubbles effectively, one has to construct $h$ which blows up in the right way for the geometry to work out and yet still satisfy the inequality \eqref{mu}. In our case, the function $h$ is identically zero in the asymptotic region.

Based on this approach, we can actually prove something stronger and more surprising than Theorem \ref{thm1}. 

\begin{thm}\label{thm2}
Let $(M^n,g)$, $3\le n\le 7$, be an asymptotically Schwarzschild manifold, not assumed to be complete or have nonnegative scalar curvature everywhere. Let $U_1$ and $U_2$ be neighborhoods of infinity with $\overline{U_2}\subset U_1$ and $D>0$. If furthermore
\begin{enumerate}
    \item $g$ has no points of incompleteness in the $D$-neighborhood of $U_1$,
    
    \item $R_g\ge 0$ in the $D$-neighborhood of $U_1$, and 
    
    \item the scalar curvature satisfies the largeness assumption 
    \begin{equation}\label{largeness}R_g> \frac{32}{D}\left(\frac{8}{D}+\frac{1}{\dist_g(U_2,\partial U_1)}\right)\end{equation}
    on $\overline{U_1}\setminus U_2$,
\end{enumerate}
then $m\ge 0$. 
\end{thm}
\begin{figure}
    \centering
    \includegraphics[width=9cm]{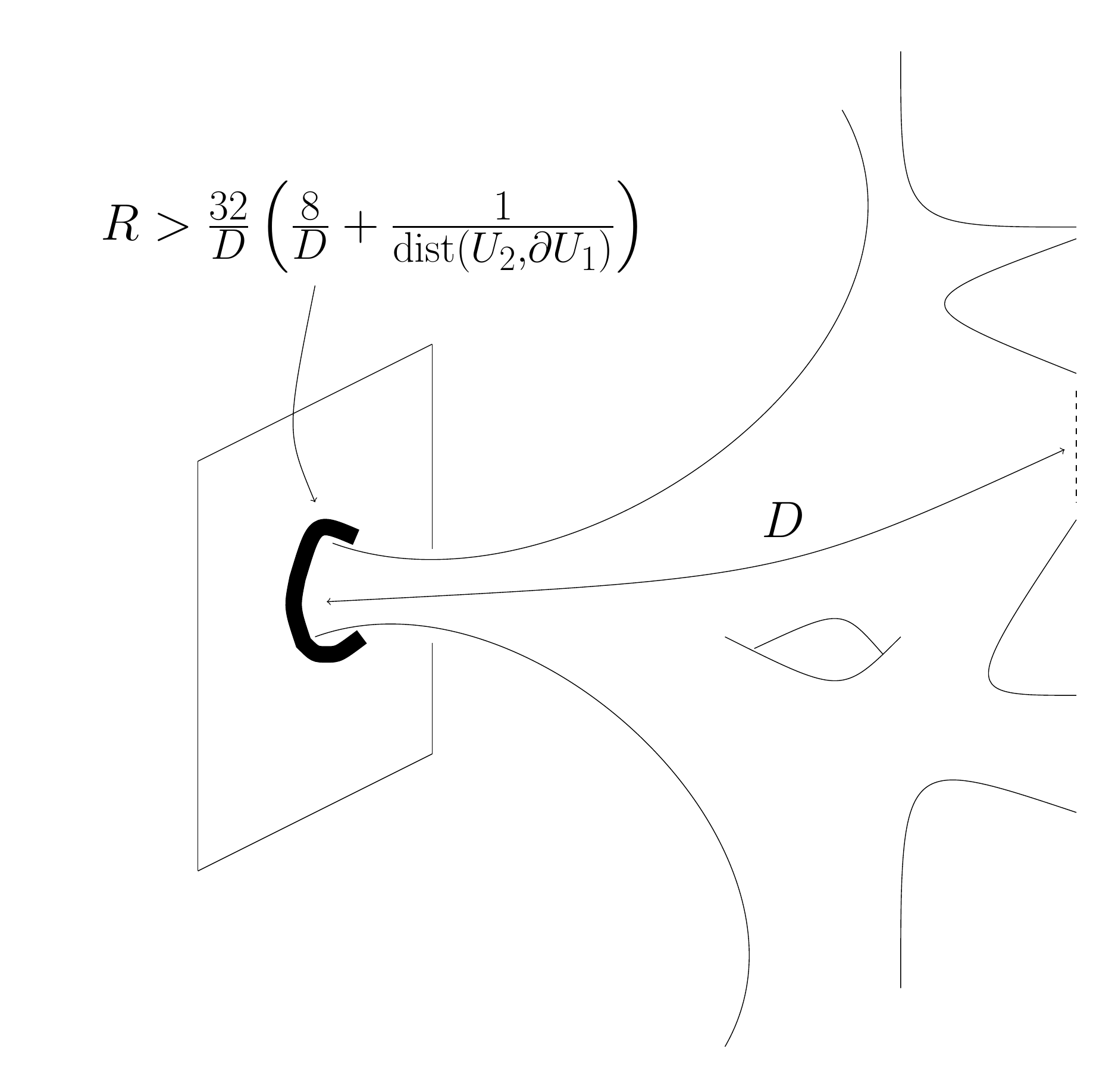}
    \caption{Theorem \ref{thm2}.}
    \label{Fig.2}
\end{figure}
This theorem is proved in Section \ref{sec:thm2} under the additional assumption that $R_g>0$ sufficiently far out on the end. This assumption is removed in Section \ref{sec:thm1} and Theorem \ref{thm1} is then obtained as a corollary.
 
\begin{rk}\label{rkbig}
(1) This theorem allows for incompleteness and negative scalar curvature as long as it is far enough away from the region in which the largeness condition \eqref{largeness} holds. 

(2) The fact that scalar curvature is only assumed to be nonnegative away from the incompleteness is particularly surprising in view of the aforementioned approaches to Theorem \ref{pmt}, which require nonnegative scalar curvature everywhere.

(3) The condition \eqref{largeness} is only needed to enforce the condition \eqref{mu} in our construction. 

(4) The prototypical $U_i$ is the region in $\mathcal E$ with $|x|> r_i$ and $r_1<r_2$. In this case, the largeness assumption \eqref{largeness} can be replaced by a lower bound involving $D$, $r_2-r_1$, and the Lipschitz constant of $|x|$ with respect to $g$, which in turn can be estimated using 
\begin{equation}|\nabla_g|x||^2\le 1+\sum_{ij}|g^{ij}-\delta^{ij}|.\end{equation}
See Remark \ref{rk:distcomp}.

(5) Theorem \ref{thm2} is interesting to consider in light of the negative mass Schwarzschild solution, which is incomplete and scalar flat and hence falls short of Theorem \ref{thm2} because of condition (3). But our theorem shows that any perturbation which increases scalar curvature (and preserves the asymptotically Schwarzschild nature) cannot ``push" the incompleteness away, so that condition (1) is not satisfied.   
\end{rk}

\textbf{Definitions and conventions.} We assume all manifolds are smooth and oriented. When discussing minimal hypersurfaces, the dimensional restriction $n\le 7$ coming from the regularity theory will be implicit. 

We use the notation $A\lesssim B$ to denote $A\le CB$, where the constant $C$ does not depend on $A$ or $B$. 

\begin{defn}\label{incomplete}
	Let $(X,d)$ be a metric space and $(\overline X,d)$ be its completion. For example, $\overline X$ can be constructed by taking appropriate equivalence classes of Cauchy sequences. A point in $\overline X\setminus X$ is called a called a \emph{point of incompleteness} for $X$. A set $S\subset X$ is said to be \emph{complete} if its closure in $X$ remains closed under the inclusion $X\to\overline X$.   
\end{defn}

\begin{defn}\label{maindef}
Let $M^n$ be a noncompact manifold with distinguished end $\mathcal E$. We say that $(M,g)$ possesses a \emph{structure of infinity} along $\mathcal E$ if $\mathcal E$ possesses no points of incompleteness and there exists a diffeomorphism 
\[\Phi:\mathcal E\to \Bbb R^n\setminus \text{ball}\]
and the coordinate norm $|x|=((x^1)^2+\cdots +(x^n)^2)^{1/2}$ diverges as we go out along the end. The set $\mathring M=M\setminus \mathcal E$ is called the \emph{core}. Note that in our definition, the core is not assumed to be compact. 

We say that $(M^n,g,\mathcal E)$ is \emph{asymptotically Schwarzschild} of \emph{mass} $m\in\Bbb R$ if in the coordinates $x^i$,
\[g_{ij}-\left(1+\frac{m}{2|x|^{n-2}}\right)^{\frac{4}{n-2}}\delta_{ij}\in C^{2,\alpha}_{n-1}(\mathcal E)\]
for some $\alpha\in (0,1)$. For the spaces $C^{k,\alpha}_\tau$, see Definition \ref{weightedhol}.

We say that $(M^n,g,\mathcal E)$ is \emph{asymptotically flat} of \emph{order} $\tau$ if 
\[g_{ij}-\delta_{ij}\in W^{2,p}_{\tau}(\mathcal E)\]
for some $p>n$ and $R_g\in L^1(\mathcal E)$. For the spaces $W^{k,p}_\tau$, see Definition \ref{weightedsob}. We will always assume $\tau>\frac{n-2}{2}$. In this case, the \emph{ADM mass} is defined by 
\begin{equation}
\label{ADM}
m=\lim_{\rho\to\infty}\frac{1}{2(n-1)|S^{n-1}|}\int_{|x|=\rho}(\partial_i g_{ij}-\partial_j g_{ii})\frac{x^j}{|x|}\,d\mathcal H^{n-1}.\end{equation}
\end{defn}

\vspace{3mm}

\textbf{Acknowledgments.} M.L.\ thanks the Gordon and Betty Moore and the John Templeton foundations for supporting the research carried out at Harvard's Black Hole Initiative. R.U.\ thanks Prof.\ T.\ Bourni, Prof.\ G.\ De Philippis, and V.\ Giri for discussions regarding Lemma \ref{curvatureestimate} and Appendix \ref{app:A} and Prof.\ R.\ Schoen for Remark \ref{Rick}. S.-T.Y.\ acknowledges the support of NSF Grant DMS-1607871.

\section{The Dirichlet problem for $\mu$-bubbles}

We first prove a general existence theorem for $\mu$-bubbles but in the setting of the Dirichlet problem. A similar situation was already considered by Yau in \cite{Y01}.

 We always assume the dimension of the ambient manifold lies between $2$ and $7$ so that we may apply the regularity theory for almost-minimizing hypersurfaces \cite{DS02, Maggi}. 
	
	\begin{thm}\label{existence}
			Let $(M^{n},g)$ be a compact orientable Riemannian manifold with boundary $\partial M=\partial_0M\amalg \partial_-M$, where $\partial_0M$ and $\partial_-M$ are closed hypersurfaces. Let $\Gamma^{n-2}\subset \partial_0M$ be a smooth hypersurface which bounds a compact set $P^n\subset\partial_0M$. Let $h$ be a real-valued smooth function on $M^\circ$, extending smoothly to $\partial_0M$, such that $h<H_{\partial_0M}$ along $\partial_0M$ and $h\to -\infty$ along $\partial_-M$. Then there exists a hypersurface $\Sigma^{n-1}\subset M$ with the following properties:
			
\begin{enumerate}
	\item $\partial \Sigma=\Gamma,$

	\item The interior of $\Sigma$ is smooth and contained in the interior of $M$, and
	
	\item The mean curvature of $\Sigma$ is equal to $h$ with respect to the outward normal in Lemma \ref{topological} below.
\end{enumerate}
	\end{thm}
	
	We view $P$ as a top-dimensional chain in $\partial_0M$ and as a codimension 1 chain in $M$.
	 Let $\Sigma^n\subset M$ be a smooth hypersurface with oriented boundary equal to $\Gamma$.
	
	\begin{lem}\label{topological}
	With the notation as above, if $\Sigma$ is homologous to $P$, then  $\Sigma\cup P=\partial U$, where $U\subset M$ is a unique open set. 
	\end{lem}
We denote this assignment by $\Sigma\mapsto U_\Sigma$.
\begin{proof}
The chain	$A=\current{\Sigma}-\current{P}$ is a cycle in $M$. Using the Mayer--Vietoris sequence, it is easy to see that $H_0(M\setminus A)$ has rank at most 2, and we want to rule out the case when it has rank 1. In this case it would be possible to find a curve intersecting $A$ exactly once. (Take a short segment orthogonal to $\Sigma$ and since $M\setminus A$ is connected, the endpoints can be joined up without further intersections with $A$.) But $A$ is homologous to zero and intersection number is a homological invariant.
\end{proof}

\begin{proof}[Proof of Theorem \ref{existence}]
	Let $\mathcal C$ be the class of smooth hypersurfaces considered in Lemma \ref{topological} and define
	\begin{equation}\label{F}\mathcal F(\Sigma)=\mathcal H^{n-1}(\Sigma)-\int_{U_\Sigma}h\,d\mathcal H^{n}.\end{equation}
	 Let $\{\Sigma_i\}\subset\mathcal C$ be a sequence of hypersurfaces such that
	 \[\lim_i\mathcal F(\Sigma_i)= \inf_\mathcal C\mathcal F.\]
	 
	 We show using a replacement argument that we can assume $\Sigma_i$ avoids a collar neighborhood of $\partial_-M$. We can foliate a collar of $\partial_-M$ by hypersurfaces $S^\tau_-=\partial_\mathrm{top} \Omega^\tau_-$ whose normal $\eta$ (pointing into the bulk, out of $\Omega^\tau_-$) satisfies $h< - \Div \eta$ for $0<\tau\le\tau_0$.  If $\Sigma_i$ enters this region, we perform a replacement and define $\Sigma'_i=\partial(U_\Sigma\setminus \Omega_-^\tau)$. Then as in Chodosh--Li \cite{CL20}, $\mathcal F(\Sigma'_i)<\mathcal F(\Sigma_i)$. So we may assume that $\Sigma_i$ avoids a collar of $\partial_-M$ and hence we may assume $|h|\lesssim 1$.  
	 
	 The condition $h<H_{\partial_0M}$ allows for a similar replacement/barrier argument near $\partial_0M$. See Yau \cite{Y01}.
	 
	 Since $M$ is compact and $h$ is bounded in the region of interest, $\mathcal F(\Sigma_i)\gtrsim -1$. We obviously have $\mathbf M(\Sigma_i)+\mathbf M(P)\lesssim 1$ by previous observations. Also, $\mathbf M(U_{\Sigma_i}) \lesssim 1.$ So by the BV compactness theorem, there exists a subsequence (still denoted by $i$) of the $U_{\Sigma_i}$'s converging in the sense of Caccioppoli sets. The limiting set will have smooth boundary away from the corner along $P$ and will have mean curvature $h$. It also satisfies a stability inequality as shown below.
\end{proof}

\begin{prop}\label{prop:2ndvar}
If $h=0$ near $\partial_0M$, then the following second variation formula holds for the functional $\mathcal F$ defined in \eqref{F}. Let $F:M\times(-\ve,\ve)\to M$ be a one-parameter family of diffeomorphisms with variation vector field $X=\partial_t F|_{t=0}$ on $\Sigma$ and acceleration $Z=\nabla_{\partial_t}\partial_t F|_{t=0}$ on $\Sigma$. Then 
\begin{equation}\label{secondvariation}\delta^2_{X,Z} \mathcal F(\Sigma)=\int_\Sigma a_{X,Z}\,d\mathcal H^{n-1}+\int_{\Sigma}d_{X,Z}\,d\mathcal H^{n-1},\end{equation}
where 
\begin{align}
a_{X,Z}&= \sum_{i=1}^{n-1}\left[|(\nabla_{e_i}X)^\perp|^2-\Rm(X,e_i,X,e_i)\right]-\sum_{i,j=1}^{n-1}\langle e_i,\nabla_{e_j}X\rangle\langle \nabla_{e_i}X,e_j\rangle\label{a}\\ &\quad +\Div_\Sigma Z+(\Div_\Sigma X)^2,\quad\text{and}\nonumber \\
\label{d}d_{X,Z} &= -h^2- X(h)\langle X, \nu\rangle- \langle Z,\nu\rangle h - \langle X,\delta_X\nu\rangle h,\end{align} where $\{e_i\}_{i=1}^{n-1}$ is an orthonormal frame on $T\Sigma$. The first term satisfies the following integration by parts formula: If $D\subset \Sigma$ is a subdomain and $X=\varphi\nu +\hat X$ and $Z=\zeta\nu+\hat Z$, then \begin{equation}\label{ibps}
   \int_D a_{X,Z}\,d\mathcal H^{n-1}=\int_D b_{X,Z}\,d\mathcal H^{n-1} + \int_{\partial D} c_{X,Z}\,d\mathcal H^{n-2},
\end{equation}
where 
\begin{align}
\label{b}b_{X,Z}&= |\nabla_\Sigma\varphi|^2-(\Ric(\nu,\nu)+|A|^2-h^2)\varphi^2+h(\zeta -2\nabla_{ \hat X}\varphi +A(\hat X ,\hat X)),\quad\text{and}  \\
\label{c}c_{X,Z}&=\langle (\Div_{\Sigma} \hat{X})\hat{X}-2\varphi S(\hat{X}) -\nabla_{\hat{X}}\hat{X}+\hat{Z},\eta\rangle, \end{align}
where $S$ is the shape operator and $\eta$ is the conormal in $\Sigma$. 

\end{prop}
\begin{proof}
It is well known \cite{Simon, Lee} that the second variation of the area functional alone is given by 
\[\delta_{X,Z}^2\mathcal H^{n-1}(\Sigma)=\int_\Sigma a_{X,Z} \,d\mathcal H^{n-1}.\]
The first variation of the bulk term is given by \cite{Maggi}
\[\delta_X\int_{U_\Sigma} h\, d\mathcal H^n=\int_\Sigma h\langle X,\nu\rangle\,d\mathcal H^{n-1}. \]
This shows that $\Sigma$ has mean curvature $h$. Varying this gives 
\[\int_\Sigma X(h)\langle X,\nu\rangle+ h\langle Z,\nu\rangle +h\langle X,\delta_X\nu\rangle  -h\langle X,\nu\rangle \langle \mathbf H,X\rangle \, d\mathcal H^{n-1}. \]
Since $h=0$ near $\partial \Sigma$, there is no boundary term involving $h$. The integration by parts formula is also standard; a careful proof can be found in \cite{Lee}.
\end{proof}

We also recall

\begin{lem}
At points where $\hat X =0$, $\delta_X\nu =-\nabla_\Sigma\varphi$. 
\end{lem}
\begin{proof}
Let $\delta_X\nu =\nabla_{\partial_t}\nu |_{t=0}$ be the variation. Using $\langle \nu,\nu\rangle =1$ along the flow, it is easy to see that $\langle\delta_X\nu ,\nu\rangle=0$. A basis to the tangent space of $\Sigma_t$ is given by $\partial_{x^i}F,$ where $x^i$ is a coordinate system on $\Sigma$. We then compute, using the equality of mixed second partials,
\begin{align*}
    \langle \partial_{x^i}F,\nabla_{\partial_t}\nu\rangle |_{t=0}&= -\langle \nabla_{\partial_t}\partial_{x^i}F,\nu\rangle|_{t=0}\\
    &=-\langle \nabla_{\partial_{x^i}}\partial_t F,\nu\rangle|_{t=0}\\
    &=-\langle \nabla_i X,\nu\rangle\\
    &= -\nabla_i\varphi.\qedhere
\end{align*}
\end{proof}
For us, it is important that merely $\delta_X\nu \perp \Sigma$.

\section{Construction of the weight}

In this section, we construct a general class of $\mu$-bubble weights in any manifold admitting a structure of infinity. Let $U_1$ and $U_2$ be neighborhoods of the distinguished infinity with compact closures in the natural one-point compactification of $\mathcal E$ and such that $\overline{U_2}\subset U_1$. We set $\rho=\dist(U_2,\partial U_1)$, which is finite. 

A model situation for this is the following: Let $\mathcal S_{r}\subset\mathcal E$ be the coordinate sphere $|x|=r$. Consider two radii $r_1<r_2$ and let $U_i$ be the set $\{x\in \mathcal E:|x|>r_i\}$.  Then $U_1\setminus\overline{U_2}$ is the annulus $\{r_1<|x|<r_2\}$. In this case, $\rho$ is not necessarily $r_2-r_1$, but see Remark \ref{rk:distcomp}.

We now state the main result of this section.

\begin{prop}\label{prop:h}
Let $(M^n,g,\mathcal E)$ possess a structure of infinity, let $U_i$ be as above, and let $D>0$. For any $\ve>0$ there exists a function $h\in C^0(M,\overline {\Bbb R})$, smooth where it is finite, with the following properties:
\begin{enumerate}
    \item $h=0$ in $U_2$,
    
    \item $h=-\infty$ outside the metric $D$-neighborhood of $U_1$, and
    
    \item $h$ satisfies the estimates
    \[h^2-2|\nabla_gh|\ge 0\quad\text{in $U_2\cup (M\setminus U_1)$},\]
    and
    \[h^2-2|\nabla_gh|\ge -\sigma(\rho,D)+ \eta(\ve) \quad\text{in $U_1\setminus U_2$},\]
    where
    \[\sigma(\rho,D)=\frac{32}{D}\left(\frac{8}{D}+\frac{1}{\rho}\right)\]
\end{enumerate}
\end{prop}

In this context, we use the notation $\eta(\ve)$ to simply mean a quantity which tends to $0$ as $\ve\to 0$. 

\begin{rk}\label{rk:distcomp}
In an asymptotically flat manifold, the metric distance $\rho$ can be estimated in terms of the coordinate distance $|x|$ using the pointwise closeness of $g_{ij}$ to $\delta_{ij}$. It is easy to see then that as $r_i\to\infty$, $\rho\to r_2-r_1$. Alternatively, we can define the cutoff function below in terms of $|x|$, and then $\sigma$ will have to involve an estimate for $|\nabla_g|x||$. 
\end{rk}

To prove Proposition \ref{prop:h} first recall the following elementary approximation result. 

\begin{lem}\label{LipApprox}
Let $(M,g)$ be a compact Riemannian manifold and $f:M\to\Bbb R$ a Lipschitz continuous function. For any $\ve>0$ there exists a smooth function $\overline f$ on $M$ such that $\|f-\overline f\|_{W^{1,\infty}(M)}<\ve.$
If $f$ is smooth on some closed set $C$, we can further demand that $f=\overline f$ on $C$.
\end{lem}
\begin{cor}\label{cutoffdecay}
Let $U_1$ and $U_2$ be open sets in a Riemannian manifold $(M,g)$ with compact closure and $\overline{U_2}\subset U_1$. Then there exists a smooth function $\vartheta$, taking values in $[0,1]$, such that $\vartheta=0$ on $U_2$, $\vartheta=1$ on $U_1^c$, and \[|\nabla_g\vartheta|\le \frac{2}{\dist_g(U_2,\partial U_1).}\]   
\end{cor}

\begin{proof}
Apply the previous lemma to the distance function to $\partial U_1$ and rescale. 
\end{proof}

\begin{proof}[Proof of Proposition \ref{prop:h}]
Our construction is motivated by consideration of solutions to the ODE $2h'=h^2$. Define 
\[\rho_0= \begin{cases} 
0 & \text{in $U_1$} \\
\dist_g(x,\partial U_1) & \text{in $M\setminus U_1$}
\end{cases}.
\]
Let $\rho_1$ be a smoothing of this function with Lipschitz constant $2$. By Lemma \ref{LipApprox}, we may assume $\rho_1$ to vanish in $U_1\setminus N_\ve (\partial U_1)$, where $N_\ve(\cdot)$ denotes the metric $\ve$-neighborhood. On $N_\ve(\partial U_1),$ $\rho_1=\eta(\ve).$  We also assume $\sup_M|\rho_0-\rho_1|<\frac 12$.

 Now define 
\[h_1=\frac{4}{\rho_1-\frac{D}{2}}.\]
Note first that $h_1=-8/D$ in $U_1\setminus N_\ve(\partial U_1)$ and $|h_1+8/D|=o(\ve)$ in $N_\ve(\partial U_1)$. Secondly, note that $h_1=-\infty$ in the region of distance $D$ away from $U_1$ in $M\setminus U_1$. Thirdly, we compute
\[|\nabla h_1|=\left|\frac{-4}{(\rho_1-\frac D2)^2}\nabla\rho_1\right|\le \frac{8}{(\rho_1-\frac D2)^2}=\frac{h^2_1}{2}.\]
Next, we define $ h=\vartheta h_1$, where $\vartheta$ is as in Corollary \ref{cutoffdecay} and compute (in the region where $\vartheta$ is variable, i.e. $U_1\setminus U_2$)
\begin{align*}
 h^2-2|\nabla h|	&=\vartheta^2h^2_1-2|\vartheta\nabla h_1+h_1\nabla\vartheta|\\
	&\ge -2\vartheta|\nabla h_1|- 2|h_1||\nabla\vartheta| \\
	&\ge -4\left(\frac{8}{D}+\eta(\ve)\right)^2-2\left(\frac{8}{D}+\eta(\ve)\right)\frac{2}{\rho}\\
	&=-\sigma(\rho,D)+\eta(\ve). \qedhere
\end{align*}
\end{proof}

\begin{rk}
By making a small perturbation of $h$, leaving the necessary inequalities satisfied, we can ensure $\partial\{|h|<\infty\}$ is a compact smooth hypersurface without boundary. We make this assumption. 
\end{rk}

We see that if $R_g>0$ everywhere and $R_g>\sigma(\rho,D)$ on $\overline{U_1}\setminus U_2$, then 
\begin{equation}
    R_g+h^2-2|\nabla_gh|>0
\end{equation}
for $\ve$ sufficiently small, which we call the \emph{(strict) $\mu$-bubble condition} for $h$. If merely
\begin{equation}
    R_g+h^2-2|\nabla_gh|\ge 0,
\end{equation}
then $h$ satisfies the \emph{weak $\mu$-bubble condition}.

\section{Proof of Theorem \ref{thm2} with an extra assumption}\label{sec:thm2}

\subsection{Construction of $\Sigma_\infty$} We single out a coordinate $x^n$ on $\mathcal E$, split the coordinates as $x=(x',x^n)$, and define 
\begin{align*}
    Z_\sigma&= \{x:|x'|\le \sigma\}\\
  C_\sigma  &=\partial(Z_\sigma\cap\{x:|x^n|\le a_0\})\\
   \Gamma_{\sigma,a} &=\{ x: |x'|=\sigma, x^n=a\},
\end{align*}
where $a_0>0$ is a constant that will be chosen momentarily. Clearly $Z_\sigma$ is mean convex for $\sigma$ sufficiently large and it is well known that the coordinate ``hyperplanes" $\{x^n=a_0\}$ are strictly mean convex for sufficiently large $a_0$ under the assumption $m<0$. We prove this in Lemma \ref{sandwich} below. We choose $a_0$ so that this is true. It follows that $C_\sigma$ is a mean convex hypersurface (with corners that can be smoothed out in a mean convex fashion).

Furthermore, we choose $P_{\sigma,a}$ to be the bottom ``hemisphere" of $C_\sigma\setminus \Gamma_{\sigma,a}$. Let $h$ be the function constructed in Proposition \ref{prop:h}. By construction, $h|_{C_\sigma} =0$. Let $\Sigma_{\sigma,a}$ be the hypersurface given by Theorem \ref{existence} with these choices. Let $\mathcal F_{\sigma}$ be the functional defined by \eqref{F}.

\begin{lem}
Fix $\sigma$. Then $a\mapsto \mathcal F_\sigma(\Sigma_{\sigma,a})$ is Lipschitz continuous for $|a|\le a_0$. 
\end{lem}
\begin{proof}
Let $a_1<a_2$ and denote $\Sigma_i=\Sigma_{\sigma,a_i}$. Define a competitor surface by 
\[\tilde \Sigma = \Sigma_2\cup \bigcup_{s=a_1}^{a_2} \Gamma_{\sigma,s}.\]
(To be rigorous, one should push this in slightly so it lies in the interior, and this only adds an $\ve$ to the following inequalities that can be sent to zero in the end, so we will ignore it.) Since $\Sigma_1$ is minimizing and $\partial \tilde\Sigma =\Gamma_{\sigma,a_1}$,
\[\mathcal F_\sigma(\Sigma_1)\le \mathcal F_\sigma(\tilde\Sigma)= \mathcal F_\sigma(\Sigma_2)+\mathcal H^{n-1} \left(\bigcup_{h=a_1}^{a_2} \Gamma_{\sigma,h}\right).\]
From asymptotic flatness, 
\[\mathcal H^{n-1}\left(\bigcup_{s=a_1}^{a_2} \Gamma_{\sigma,s}\right)\lesssim |a_2-a_1|.\]
Interchanging $a_1$ and $a_2$ and running the same argument gives altogether
\[|\mathcal F_\sigma(\Sigma_{1})-\mathcal F_\sigma(\Sigma_2)|\lesssim |a_2-a_1|.\qedhere\]
\end{proof}

\begin{figure}
    \centering
    \includegraphics[width=8.5cm]{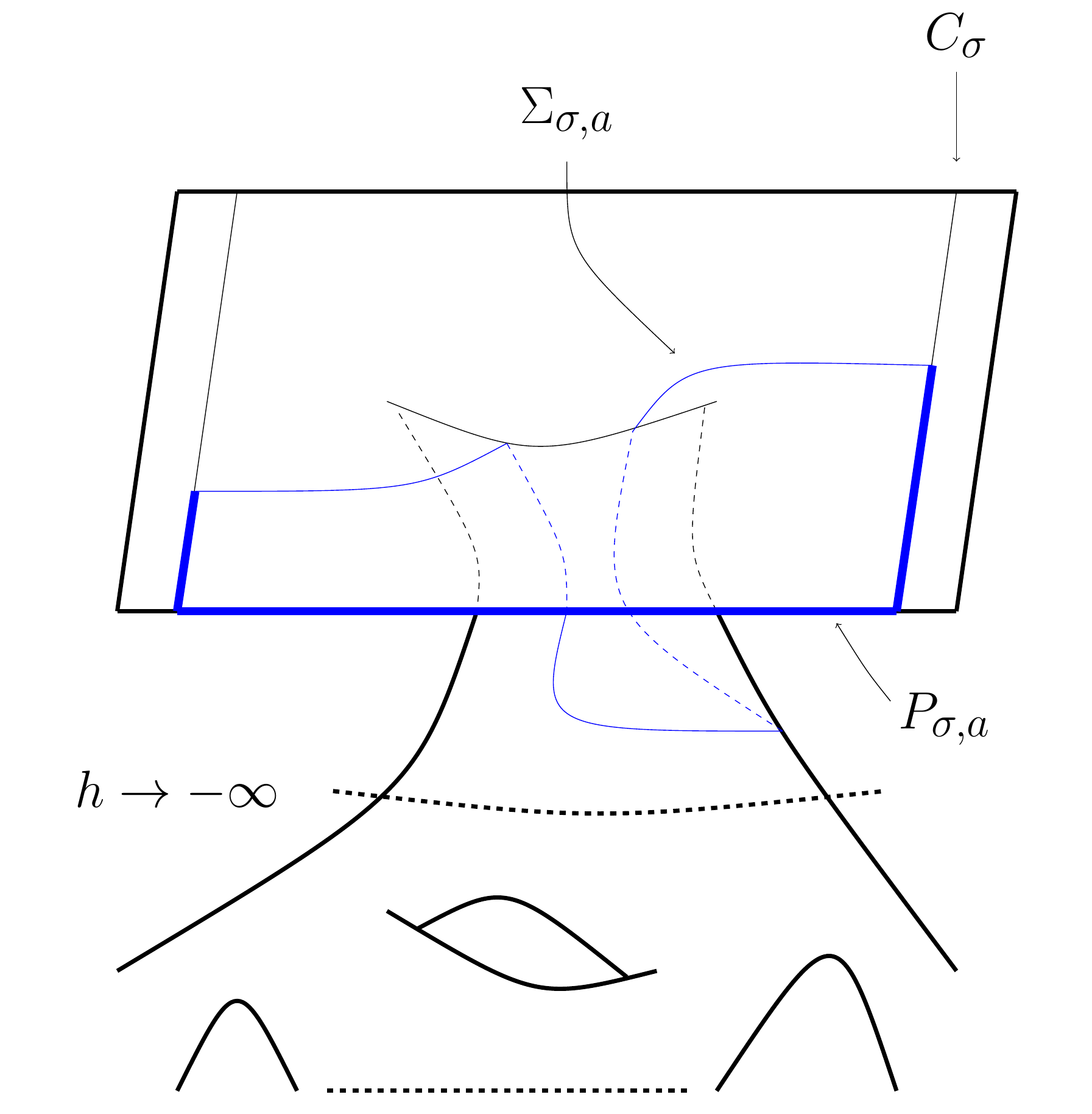}
    \caption{}
    \label{Fig.2}
\end{figure} 

\begin{lem}\label{lem:heights}
$F(\sigma)=\min\{F_{\sigma}(\Sigma_{\sigma,a}):|a|\le a_0\}$ is achieved by an $a\ne \pm a_0$.
\end{lem}
\begin{proof}
Any surface with boundary $\Gamma_{\sigma,\pm a_0}$ can be pushed up or down slightly to decrease its area, see the original proof of Schoen \cite{SchoenPMT}. Since our functionals $\mathcal F_\sigma$ are just the area functional in a neighborhood of $C_\sigma$, this argument carries through to decrease $\mathcal F_\sigma$. 
\end{proof}

If $a(\sigma)$ is the minimizing $a$ produced by the previous lemma, we set  $\Sigma_\sigma=\Sigma_{\sigma,a(\sigma)}.$ We now study the limit $\sigma\to \infty$. 

It is now convenient to introduce some more notation. Let 
$H=\{x^n=0\}$ be a hyperplane in $\Bbb R^n$ and $p:\Bbb R^n\to H$ the orthogonal projection. For the rest of this section we delete from $M$ everything behind one of the surfaces $S_-^\tau$ in the proof of Theorem \ref{existence}. The point is that each $\Sigma_{\sigma,a}$ we consider lies in this region, which is compact if we one-point compactify $\mathcal E$. Let $R_0$ be a radius so that $h=0$ for any $x\in \mathcal E$ with $|x|\ge R_0$.  We define $\mathcal E'=\mathcal E\cap \{|x'|\ge R_0\}$, so $\mathcal E'$ is the end minus a cylindrical region containing the set where $h$ is variable. We can clearly define $p:\mathcal E'\to H':=H\setminus B_{R_0}(0)$, where we now view $H'$ as a submanifold of $M$.  

\begin{lem}\label{surjective}
The orthogonal projection $p:\Sigma\cap\{R_0<|x'|<\sigma\}\to H'\cap \{R_0<|x'|<\sigma\}$ is surjective.
\end{lem}
\begin{proof}
Delete from $\Sigma_\sigma$ everything outside of $\mathcal E'$; denote the resulting current by $\Sigma'$. Note that $\partial\Sigma'=\Gamma_{\sigma,a(\sigma)}$ as currents in $\mathcal E'$. Let $T$ be the multiplicity 1 current in $\mathcal E'$ whose underlying set is given by $\{(x',x^n):|x'|\ge \sigma, x^n=a(\sigma)\}.$ Then $\partial(T+\Sigma')=0$ and \[\partial p_\#(T+\Sigma')=p_\#\partial (T+\Sigma')=0.\] By the constancy theorem \cite[Theorem 26.27]{Simon}, there exists $c\in \Bbb R$ such that \[p_\# (T+\Sigma')=c\current{H'}.\]
Looking at radii larger than $\sigma$ implies $c=1$. Now it is easy to see that if $\Sigma'$ cannot miss some fiber $p^{-1}(x')$ with $|x'|<\sigma$. 
\end{proof}

\subsection{Asymptotics of $\Sigma_\infty$}\label{sec:asymptotics}

We work with the following functional spaces. 

\begin{defn}[Weighted H\"older spaces]\label{weightedhol}
Let $\mathcal E$ be a Euclidean end of a manifold $(M^n,g)$ admitting a structure of infinity as per Definition \ref{maindef}. For $k\in \Bbb N_0$, $\alpha\in (0,1)$, and $\tau \in\Bbb R$, we define the \emph{weighted H\"older space} $C^{k,\alpha}_\tau(\mathcal E)$ as the set of all functions $u\in C^{k,\alpha}_\loc$ for which the norm 
\[\|u\|_{C^{k,\alpha}_\tau}=\sum_{i=0}^k\sup_{x\in\mathcal E}||x|^{i+\tau}\partial^i u|+\sup_{x\in\mathcal E}|x|^{k+\alpha+\tau}[\partial^ku]_{C^\alpha(B_1(x))}\] is finite. Note that with our notation, $\tau>0$ corresponds to \emph{decay}, which is opposite from \cite{Lee}. Under pointwise multiplication, $C^{2,\alpha}_{\mu}\cdot C^{2,\alpha}_\nu \subset C^{2,\alpha}_{\mu+\nu}$, continuously. We will frequently use the fact that $|x|^{-\tau}\in C^{k,\alpha}_{\tau}$ for any $k$ and $\alpha$. By definition, $\partial(C^{k,\alpha}_\tau)\subset C^{k-1,\alpha}_{\tau +1}$ continuously. 
\end{defn}

Finally, we use the notation $A*B$ to denote contractions and $\Bbb R$-linear combinations of quantities. The decay rates can be easily read off from this notation. 

Now we study the convergence of the $\Sigma_\sigma$'s and the asymptotic behavior of the limiting hypersurface. 

\begin{thm}\label{maindecay}
There exists a sequence $\sigma_j\to\infty$ such that $\Sigma_{\sigma_j}$ converges in $C^\infty_\loc$ to a complete, smooth hypersurface $\Sigma_\infty$ with mean curvature $h$. Furthermore, there exist constants $a_\infty, c_\infty$ such that for any $\ve \in (0,1)$, $\Sigma_\infty\cap \mathcal E'$ can be written as a Euclidean graph $x^n=u(x')$ for some smooth function $u$ satisfying 
\begin{align*}
    u(x')&-a_\infty \in C^{3,\alpha}_{\ve}(H')\quad\text{$n=3$},\\
    u(x')&-\left(a_\infty+\frac{c_\infty}{|x'|^{n-3}}\right)\in C^{3,\alpha}_{n-3+\ve}(H') \quad\text{$n\ge 4$},
\end{align*}
where now $\alpha$ is the minimum of the $\alpha$ in Definition \ref{maindef} and $\frac 1n$.
\end{thm}

The proof occupies the rest of this section. The locally smooth convergence of the hypersurfaces $\Sigma_{\sigma_j}$ for some sequence $\sigma_j\to\infty$ follows from standard theory \cite{DS02, Maggi, Simon} as soon as we note that a local mass bound,  follows from the fact that it minimizes a prescribed mean curvature-type functional in a compact subset of the core and the area functional in the asymptotic region. Strictly speaking, we must also show that the mass is locally bounded below so that $\Sigma_\infty\ne 0$. This follows straightforwardly from Lemma \ref{surjective}.

The following is the closest we can get to harmonic asymptotics without a density theorem. It says that we still have conformal flatness modulo strictly lower order terms, which is crucial for our analysis. 

\begin{prop}\label{table}
Any asymptotically Schwarzschild metric $g$ can be written in the form 
\[g=\psi^\frac{4}{n-2}\overline g, \]
where 
\[\psi(x)=1+\frac{m}{2|x|^{n-2}}\]
is the conformal factor of the Schwarzschild metric and 
\[\overline g_{ij}-\delta_{ij}\in C^{2,\alpha}_{n-1}.\]
With respect to these metrics, we have the following decay rates for components of various quantities in the asymptotic coordinate system:
\begin{table}[h]\label{decays}

\begin{tabular}{l|l|l|}
\cline{2-3}
                                    & $g$ & $\overline g$ \\ \hline
 \multicolumn{1}{|l|}{$\Delta \psi$} &  $C^{1,\alpha}_{2n-2}$   &  $C^{1,\alpha}_{2n-1}$             \\ \hline
\multicolumn{1}{|l|}{$\Gamma$}      & $C^{1,\alpha}_{n-1}$    &  $C^{1,\alpha}_n$             \\ \hline
\multicolumn{1}{|l|}{$\mathrm{Rm}$} & $C^{0,\alpha}_{n}$    &    $C^{0,\alpha}_{n+1}$            \\ \hline
\multicolumn{1}{|l|}{$\mathrm{Rc}$} & $C^{0,\alpha}_{n}$     &    $C^{0,\alpha}_{n+1}$            \\ \hline
\multicolumn{1}{|l|}{$R$}           & $C^{0,\alpha}_{n+1}$    &    $C^{0,\alpha}_{n+1}$           \\ \hline
\end{tabular}
\end{table}
\end{prop}

\begin{rk}
$R_g\in C^{0,\alpha}_{n+1}$ implies $R_g\in L^1(\mathcal E).$ This improved decay is not true for general asymptotically flat metrics of order $n-2$. 
\end{rk}

\begin{proof}
By assumption, $g_{ij}-\psi^\frac{4}{n-2}\delta_{ij}\in C^{2,\alpha}_{n-1},$ so
\[\overline g_{ij}-\delta_{ij}=\psi^{-\frac{4}{n-2}}g_{ij}-\delta_{ij} \in \psi^{-\frac{4}{n-2}}\cdot C^{2,\alpha}_{n-1}\subset C^{2,\alpha}_{n-1}.\]

To compute $\Delta \psi$, we compute $\Delta|x|^{-(n-2)}$ explicitly to leading order. The sign of this leading term will be particularly useful in Section \ref{sec:thm1}. Using elementary matrix expansions, we have 
\begin{align*}
   g_{ij}&-\delta_{ij}-\frac{2m}{n-2}|x|^{-(n-2)}\delta_{ij}\in C^{2,\alpha}_{n-1},\\
   g^{ij} &- \delta^{ij}+\frac{2m}{n-2}|x|^{-(n-2)}\delta^{ij}\in C^{2,\alpha}_{n-1}, \\
   \det({g_{ij}}) &- 1-\frac{2nm}{n-2}|x|^{-(n-2)}\in C^{2,\alpha}_{n-1}, \\
   \sqrt{\det(g_{ij})} &- 1-\frac{nm}{n-2}|x|^{-(n-2)} \in C^{2,\alpha}_{n-1}.
\end{align*}
Multiplying the second and fourth lines gives
\[\sqrt{\det (g_{ij})}g^{ij}-\delta_{ij}-m|x|^{-(n-2)}\delta_{ij}\in C^{2,\alpha}_{n-1}.\]
By the useful formula $\partial_j |x|^{-p}= -p |x|^{-(p+2)}x^j$, we compute
\begin{align}\label{step1}
\Delta_g |x|^{-(n-2)}    &=\frac{1}{\sqrt{\det g}}\sum_{ij}\partial_i \left(\sqrt{\det g}g^{ij}\partial_j |x|^{-(n-2)}\right)\\
&=\frac{-(n-2)}{\sqrt{\det g}}\sum_{ij}\partial_i \left(|x|^{-n}x^i+m|x|^{-(2n-2)}x^i+C^{2,\alpha}_{2n-2}\right) \nonumber\\
&= m(n-2)^2 |x|^{-(2n-2)}+C^{1,\alpha}_{2n-1}.\nonumber 
\end{align}
The same calculation for $\overline g$ gives $\Delta_{\overline g}|x|^{-(n-2)}\in C^{1,\alpha}_{2n-1}$. 

The decay rates of the Christoffel symbols can be computed from the schematic equation $\Gamma = g^{-1}*\partial g$ and for the Riemann and Ricci curvatures, from $ \mathrm{Rm}\sim \mathrm{Rc} = \partial \Gamma +\Gamma *\Gamma $. This also yields a decay rate $R_g\in C^{0,\alpha}_{n}$ and $R_{\overline g}\in C^{0,\alpha}_{n+1}$. We now improve the decay rate for $R_g$ using the conformal structure. Using the standard formula for scalar curvature under a conformal change (where $L$ is the conformal Laplacian),
\begin{align*}
    R_g&=\psi^{-\frac{n+2}{n-2}}L_{\overline g}\psi\\
    &= C^{2,\alpha}_0\cdot \left(-4\frac{n-1}{n-2} \Delta_{\overline g}\psi+R_{\overline g}\psi\right)\\
    &=C^{2,\alpha}_0\cdot (C^{1,\alpha}_{2n-1}+C^{0,\alpha}_{n+1}\cdot C^{0,\alpha}_0)\in C_{n+1}^{0,\alpha}.\qedhere
\end{align*}
\end{proof}

We have an initial curvature estimate for $\Sigma_\infty$. This will ultimately be improved in Lemma \ref{improved} below. 

\begin{lem}\label{curvatureestimate}
The asymptotically minimal hypersurface $\Sigma_\infty$ enjoys the curvature decay
\[|A|= O(|x'|^{-1})\] with respect to the metric $g$.
\end{lem}
\begin{proof}
After rescaling, such a curvature estimate follows from work of Schoen--Simon, see either \cite{SSstable} or \cite{SSelliptic}. To apply \cite{SSstable}, we note that the injectivity radius is controlled from below by \cite{CGT} and the closeness of the metric to flat. In general, it is not possible to obtain anything better than $C^{0,\alpha}$ control on the metric components in normal coordinates given a $C^{2,\alpha}$ bound \cite{DTK}. However, we can introduce $C^{2,\alpha}$ harmonic coordinates about every point of interest and these will satisfy the assumptions (1.2)-(1.6) in \cite{SSstable}. 

We give another proof of this result in Appendix \ref{app:A} via a ``universal" curvature estimate for area minimizing hypersurfaces in manifolds with bounded sectional curvature.
\end{proof}

\begin{lem}[Comparison lemma {\cite[Lemma 2.4]{EichmairMetzger}}]\label{comparison}
In an asymptotically Schwarzschild manifold, there exists a constant $C=C(g)$ with the following property. For any hypersurface $\Sigma$, the second fundamental forms relative to $g$ and $\delta$ satisfy
\[|A_g-A_\delta|_g\le  C|x|^{-(n-1)}(|x||A_g|_g+1).\]
For the mean curvature, we have the comparison
\[|H_g-H_\delta|\le C|x|^{-(n-1)}( |x||A_g|_g+1).\]
\end{lem}

For the proof, we require a simple lemma comparing the norms defined by the metrics $g$ and $\delta$. 

\begin{lem}\label{norms}
In an asymptotically Schwarzschild manifold, the norms defined by $g$ and $\delta$ satisfy
\[||X|_g-|X|_\delta|\lesssim |x|^{-(n-2)}|X|_\delta\lesssim |x|^{-(n-2)}|X|_g,\]
where $X$ is any tangent vector. 
\end{lem}
\begin{proof}
From the decay of $g_{\ij}-\delta_{ij}$, 
\[\left|\frac{|X|^2_g}{|X|^2_\delta}-1\right|\lesssim |x|^{-(n-2)},\] Factoring the left-hand side and dividing through gives
\[\left|\frac{|X|_g}{|X|_\delta}-1\right|\lesssim |x|^{-(n-2)}\left(\frac{|X|_g}{|X|_\delta}+1\right)^{-1},\]
Clearly
\[\frac{|X|_g}{|X|_\delta}\approx 1,\]
so that 
\[\left|\frac{|X|_g}{|X|_\delta}-1\right|\lesssim|x|^{-(n-2)},\]
which implies the claim. 
\end{proof}

To prove Lemma \ref{comparison}, we use the formulas for a hypersurface $\Sigma$ defined as a level set $\{\zeta=0\}$ for some smooth function $\zeta$ satisfying $d\zeta\ne 0$. It is standard \cite{Petersen} that the normal with respect to $g$ is given by 
\begin{equation}\label{normal}
    \nu_g=\frac{\nabla_g\zeta}{|d\zeta|_g},
\end{equation}
the second fundamental form is given by 
\begin{equation}\label{sff}
A_g(X,Y)=\frac{\nabla^2_g\zeta(X,Y)}{|d\zeta|_g},\quad X,Y\in T\Sigma,\end{equation}
and the mean curvature is given by 
\begin{equation}\label{meancurvature}
H_g=\Pi^{ij}_g\frac{(\nabla^2_g\zeta)_{ij}}{|d\zeta|_g},\end{equation}
where 
\begin{equation}\label{projection}\Pi_g= g^{-1}-\nu_g\otimes \nu_g\end{equation}
is the projection matrix $T^*M\to T\Sigma$. We emphasize the dependence on the metric in these formulas because we deal with three different metrics in this section.

\begin{proof}[Proof of Lemma \ref{comparison}]
Consider $\Sigma$ as a level set $\{\zeta=0\}$. For $X$ and $Y$ tangent to to $\Sigma$, we have
\[(A_g-A_\delta)(X,Y)= \left(\frac{1}{|d\zeta|_g}-\frac{1}{|d\zeta|_\delta}\right)\partial_{ij}\zeta X^iY^j- \Gamma^{k}_{ij}\frac{\partial_k \zeta}{|d\zeta|_g}X^i Y^j.\]
Using Lemma \ref{norms},
\[\left|\frac{1}{|d\zeta|_g}-\frac{1}{|d\zeta|_\delta}\right|= \frac{||d\zeta|_\delta -|d\zeta|_g|}{|d\zeta|_g|d\zeta|_\delta}\lesssim \frac{|x|^{-(n-2)}}{|d\zeta|_g}.\]
Now we estimate
\begin{align*}
    |(A_g-A_\delta)(X,Y)| &\le C |x|^{-(n-2)}\left|\frac{\partial_{ij}\zeta X^i Y^j}{|d\zeta|_g}\right|+ C|\Gamma||X|_g||Y|_g\\
   & \lesssim |x|^{-(n-2)}\left|A_g(X,Y)+\Gamma^k_{ij}\frac{\partial_k\zeta}{|d\zeta|_g}X^iY^j\right|+|x|^{-(n-1)}|X|_g|Y|_g\\
   &\lesssim |x|^{-(n-2)}|A_g(X,Y)| + (|x|^{-(n-1)}+|x|^{-n(n-1)})|X|_g|Y|_g\\
   &\lesssim |x|^{-(n-1)}(|x||A_g(X,Y)|+|X|_g|Y|_g),
\end{align*}
which completes the proof. 
\end{proof}

\begin{cor}
The mean curvature of $\Sigma_\infty$ is $O(|x'|^{-(n-1)})$ relative to the Euclidean metric. 
\end{cor}
\begin{proof}
$H_g=0$ for $\Sigma_\infty$ on $\mathcal E$ and $|x||A_g|_g$ is bounded.  
\end{proof}

This decay on $H_\delta$ enables the use of Allard's regularity theorem.

\begin{prop}\label{AllardCorollary}
Let $\Sigma_\infty$ be the asymptotically minimal hypersurface considered in Proposition \ref{maindecay}. Choosing $R_0$ perhaps larger, we can write $\Sigma_\infty$ as the graph of a function $u\in C^{1,\gamma}_0(H')$, where $\gamma=\frac{1}{n}$. 
\end{prop}
\begin{proof}
We apply Allard's regularity theorem in the form \cite[Theorem 23.1]{Simon} with $p=n$ and $T=\{x^n=0\}\cong \Bbb R^{n-1}$. The $L^n$ norm of the mean curvature is estimated by 
\[
   \left( \int_{B_\rho(x)\cap \Sigma_\infty} |H_\delta|^n\,d\mathcal H^{n-1}\right)^\frac{1}{n} \rho^\frac{1}{n}\lesssim \rho^{-(n-2)}\le \ve_n,\]
   where $\ve_n$ is the constant appearing in Allard's theorem. Here we use $|H_\delta|\lesssim |x'|^{-(n-1)}$, which is crucial. The excess is estimated using \cite[Lemma 22.2]{Simon} and the height bound, which implies $|x|\lesssim |x'|$ on $\Sigma_\infty$. For details, see \cite[4.2.2]{Kuwert}.
  \end{proof}

Now the strategy is to write $H_g=0$ as a prescribed curvature equation in Euclidean space and to bootstrap this decay rate.  We define $\zeta(x',x^n)=u(x')-x^n$, so that $\Sigma_\infty=\{\zeta=0\}$. It is convenient to use Greek indices running from $1$ to $n$ and Latin from $1$ to $n-1$. We see immediately that $\partial_\mu\zeta=(\partial_iu,-1)$ and $\partial_{\mu\nu}\zeta$ is only nonzero when $\mu\nu=ij$, in which case it equals $\partial_{ij}u$. Note that with this definition, the normal coming from \eqref{normal} points down, but this is of no consequence since $H_g=0$ is independent of this choice. 

\begin{prop}\label{bigeqn}
 The function $u$ satisfies
 \begin{equation}\label{maineqn}\Pi_{\delta}^{ij}\partial_{ij}u= B^{ij}\partial_{ij}u +\Pi_{\overline g}*\Gamma_{\overline g}*\partial\zeta -2\frac{n-1}{n-2} \overline g^{\mu\nu}\partial_\mu\zeta\partial_\nu\log \psi,\end{equation}
 where, schematically,
\[B=p+\left(\frac{1}{|d\zeta|^2_{\overline g}}-\frac{1}{|d\zeta|_\delta^2}\right)\partial u*\partial u +\frac{p*\partial\zeta*\partial\zeta}{|d\zeta|_{\overline g}^2} +\frac{p*p*\partial\zeta*\partial\zeta}{|d\zeta|^2_{\overline g}},\]
and $p^{\mu\nu}=\overline g^{\mu\nu} -\delta^{\mu\nu}$.
\end{prop}
\begin{proof}
Using the conformal transformation properties of mean curvature \cite[page 33]{Lee}, we get 
\[H_{\overline g}= -2\frac{n-1}{n-2} \,\overline \nu(\log \psi),\]
since $H_g=0.$ Using equations \eqref{normal} and \eqref{meancurvature}, we get 
\[\Pi_{\overline g}^{\mu\nu}(\partial_{\mu\nu}\zeta -\Gamma_{\overline g\,\mu\nu}^\kappa\partial_\kappa\zeta)=-2\frac{n-1}{n-2}\grd_{\overline g}\zeta(\log \psi)\]
To obtain our conclusion, we simply have to compute $\Pi_{\overline g}^{\mu\nu}\partial_{\mu\nu}\zeta-\Pi^{\mu\nu}_\delta \partial_{\mu\nu}\zeta$:

\begin{align*}
    \Pi^{ij}_{\overline g} \partial_{ij}u-\Pi_\delta^{ij}\partial_{ij} u&=\left(\delta^{ij}+p^{ij}-\frac{(\delta^{i\alpha}+p^{i\alpha})(\delta^{j\beta}+p^{j\beta})\partial_\alpha\zeta\partial_\beta\zeta}{|d\zeta|_{\overline g}^2}\right)\partial_{ij}u\\
    &\quad -\left(\delta^{ij}-\frac{\delta^{i\alpha}\delta^{j\beta}\partial_\alpha\zeta\partial_\beta\zeta}{|d\zeta|_\delta^2}\right)\partial_{ij}u\\
    &=p^{ij}\partial_{ij}u-\left(\frac{1}{|d\zeta|^2_{\overline g}}-\frac{1}{|d\zeta|_\delta^2}\right)\delta^{i\alpha}\delta^{j\beta}\partial_\alpha\zeta\partial_\beta\zeta \partial_{ij}u+\frac{p*\partial\zeta*\partial\zeta}{|d\zeta|_{\overline g}^2}*\partial^2u\\
    &\quad+\frac{p*p*\partial\zeta*\partial\zeta}{|d\zeta|^2_{\overline g}}*\partial^2 u\\
    &=\left(p+\left(\frac{1}{|d\zeta|^2_{\overline g}}-\frac{1}{|d\zeta|_\delta^2}\right)\partial u*\partial u +\frac{p*\partial\zeta*\partial\zeta}{|d\zeta|_{\overline g}^2} +\frac{p*p*\partial\zeta*\partial\zeta}{|d\zeta|^2_{\overline g}}\right)*\partial^2 u.
\end{align*}
\end{proof}

Applying these computational tricks to the function $\zeta(x)=x^n$ gives
\begin{lem}\label{sandwich}
In an asymptotically Schwarzschild manifold with negative mass, there exists an $a_0>0$ such that for any $a\ge a_0$, the ``hyperplane" $\{x^n=a\}$ is strictly mean convex with respect to the upward normal and $\{x^n=-a\}$ is strictly mean convex with respect to the downward normal.
\end{lem}
\begin{proof}
Using the above equations, and computing relative to the upward normal, the mean curvature of a constant $x^n$ hyperplane is given by
\[H_g=-(n-1)m\frac{x^n}{|x|^n}+O(|x|^{-n}).\qedhere\]
\end{proof}

We now redefine $\alpha$ to be the minimum of the original $\alpha$ and $\frac{1}{n}$. So by Proposition \ref{AllardCorollary}, $u\in C^{1,\alpha}_0$. The following is very useful (it will be used implicitly in almost every single line later) and is easy to prove. 

\begin{lem}[\emph{Hilfsaussage} {\cite[page 82]{Kuwert}}]\label{hilfsaussage}
Assume $u\in C^{1,\alpha}_0(H')$. Given a function $f\in C^2_\tau(\mathcal E)$, its restriction to the graph of $u$, $\tilde f(x')=f(x',u(x'))$, lies in $C^{1,\alpha}_\tau(H')$. The same holds for $f\in C^{1,\alpha}_\tau(\mathcal E)$.\end{lem}

\begin{lem}\label{decays}
The assumption $u\in C^{1,\alpha}_0$ implies the following decay rates:
\begin{enumerate}
    \item $\partial\zeta\in C^{0,\alpha}_0$
    
    \item $\frac{1}{|d\zeta|^2_{\overline g}}\in C^{0,\alpha}_0$
    
    \item $\frac{1}{|d\zeta|^2_{\overline g}}-\frac{1}{|d\zeta|_\delta^2}\in C^{0,\alpha}_{n-1}$
    
    \item $\delta^{ij}-\Pi_\delta^{ij}\in C^{0,\alpha}_{2}$
    
    \item $\Pi_{\overline g}\in C^{0,\alpha}_0$
    
    \item $\Gamma_{\overline g}(x',u(x'))\in C^{1,\alpha}_n$
\end{enumerate}
\end{lem} 
\begin{proof}
(1) follows from Proposition \ref{AllardCorollary}. (2) follows (1) and the decay of the metric. (3) follows from the simple calculation
\[\frac{1}{|d\zeta|_{\overline g}^2}-\frac{1}{|d\zeta|^2_\delta}=\frac{|d\zeta|_\delta^2-|d\zeta|^2_{\overline g}}{|d\zeta|_\delta^2|d\zeta|^2_{\overline g}}=-\frac{p(d\zeta,d\zeta)}{|d\zeta|_\delta^2|d\zeta|^2_{\overline g}} \in C^{0,\alpha}_{n-1}.\]
(4) follows from the definition,
\[\delta^{ij}-\Pi_\delta^{ij} = \frac{\partial_iu\partial_j u}{1+|\partial u|^2},\]
and $\partial u\in C^{0,\alpha}_1$. Points (5) and (6) follow using these same ideas and the \emph{Hilfsaussage}. 
\end{proof}

To obtain the desired asymptotic expansion for $u$, we use the following theorem of Meyers (for a proof, see \cite[Appendix A]{Lee}). 

\begin{thm}[Meyers \cite{Meyers}]\label{Meyers} Let $(M^{n-1},g)$ be an asymptotically flat manifold. Suppose $Lu=f$, where $u$ is a bounded smooth function and $Lu=a^{ij}\partial^2_{ij} u + b^i\partial_i u+cu$ is a linear elliptic operator satisfying the standard conditions $a^{ij}-\delta^{ij}\in C^{k,\alpha}_{\ve}$, $b^i\in C^{k,\alpha}_{1+\ve}$, $c\in C^{k,\alpha}_{n-1+\ve}$, and $f\in C^{k,\alpha}_{n-1+\ve}$ for some $\ve\in (0,1)$. Then there exist constants $h_\infty,c_\infty$ such that 
\begin{align*}
    u(x)&-a_\infty \in C^{k+2,\alpha}_{\ve}\quad\text{$n=3$},\\
    u(x)&-\left(a_\infty+\frac{c_\infty}{|x|^{n-3}}\right)\in C^{k+2,\alpha}_{n-3+\ve} \quad\text{$n\ge 4$}.
\end{align*}
\end{thm}
We emphasize that the theorem is stated here for $(n-1)$-dimensional manifolds. We can finally prove the main result of this section.

\begin{proof}[Proof of Proposition \ref{maindecay}]
We rewrite \eqref{maineqn} in the form
\begin{equation}
    (\Pi_\delta^{ij}-B^{ij})\partial_{ij}u+(\Pi_{\overline g}*\Gamma_{\overline g}+\partial\log \psi)*\partial u= \Pi_{\overline g}*\Gamma_{\overline g} +\overline g*\partial_n \log\psi, 
\end{equation}
neglecting numerical constants multiplying the various terms. Comparing with Meyers' theorem, the only thing that needs to be checked is that $\partial_n\log\psi\in C^{1,\alpha}_{n-1+\ve}$. Naively, we only get a decay rate of $n-1$, but the $x^n$ derivative enjoys extra decay when restricted to the graph:
\[\partial_n\psi =c\frac{u}{|x|^n}\in C^{1,\alpha}_{n}.\]
Applying Myers' theorem with $k=0$ gives a constant $a_\infty$ so that $v:=u-a_\infty\in C^{2,\alpha}_{n-3}$. To improve the decay further we wish to use the theorem with $k=1$, but this requires improving the decay of the coefficients first. It is easy to see that the (newfound) decay of $\partial^2 u$ now gives $C^{1,\alpha}$ decay for every quantity in Lemma \ref{decays} which involves $u$. So we may apply Meyers' theorem with $k=1$ and conclude our result.
\end{proof}

\begin{cor}\label{SigmaAF}
When $n\ge 4$, the induced metric $\gamma$ on $\Sigma_\infty$ is asymptotically flat with zero mass. More precisely, the projection $\Sigma_\infty\to H'$ induces coordinates $y^i$ on $\Sigma_\infty$ so that
\[\gamma_{ij}=\delta_{ij}+C^{2,\alpha}_{n-2}(H').\]
\end{cor}
\begin{proof}
We use the standard graphical coordinates $y^i$ associated to $x'\mapsto (x',u(x'))$. The coordinate vector fields are given by 
\[\frac{\partial}{\partial y^i}=\frac{\partial}{\partial x^i}+\frac{\partial u}{\partial x^i}\frac{\partial}{\partial x^n},\]
so we compute 
\begin{align*}
  \gamma(\partial_{y^i},\partial_{y^j}) &= g(\partial_{y^i},\partial_{y^j})\\
    &=\psi^\frac{4}{n-2}(x',u(x'))\left(\delta_{ij}+\partial_iu\partial_ju+ (\overline g-\delta)(\partial_{y^i},\partial_{y^j})\right)\\
    &=(1+C^{2,\alpha}_{n-2})^\frac{4}{n-2}\left(\delta_{ij}+C^{2,\alpha}_{n-2}\cdot C^{2,\alpha}_{n-2}+C^{2,\alpha}_{n-1}\right)\\
    &= \delta_{ij}+C^{2,\alpha}_{n-2}. 
\end{align*}
Now the ADM mass integrand \ref{ADM} associated to $\gamma$ decays like $O(|y|^{-(n-1)})$, whereas the coordinate spheres have area growing like $O(|y|^{n-2})$ (recall $\dim \Sigma_\infty=n-1$), so the mass is zero. 
\end{proof}

We remark for later that Theorem \ref{maindecay} gives improved decay of the second fundamental form.

\begin{lem}\label{improved}
The asymptotically minimal hypersurface $\Sigma_\infty$ enjoys the curvature decay 
\[|A|=O(|x'|^{-(n-1)})\]
with respect to the metric $g$.
\end{lem}
\begin{proof}
By Lemma \ref{comparison} and Lemma \ref{curvatureestimate}, it suffices to show that $|A_\delta|=O(|x'|^{-(n-1)})$. Since $\partial u$ is bounded, $|A_\delta|\approx |\partial^2 u|=O(|x'|^{-(n-1)})$.
\end{proof}

\subsection{Strong stability of $\Sigma_\infty$}   By construction, the $\mu$-bubbles $\Sigma_{\sigma,a}$ satisfy the stability inequality with respect to variations which vanish on the boundary. But as a result of our height picking, Lemma \ref{lem:heights}, we gain an additional \emph{vertical stability} property. This leads to an improved stability inequality on $\Sigma_\infty$ which holds not only for compactly supported functions (and limits thereof), but functions of the form $\alpha+\varphi$, where $\varphi$ can be approximated by compactly supported functions and $\alpha$ is a possibly nonzero constant. This is called the \emph{strong stability inequality} \cite{SchoenPMT} and to state it precisely we need one more definition.

\begin{defn}[Weighted Sobolev Spaces] \label{weightedsob}
Let $\mathcal E$ be a Euclidean end of a manifold $(M^n,g)$ admitting a structure of infinity per definition \ref{maindef}. For $p\ge 1$ an $\tau\in \Bbb R$, we define the \emph{weighted Lebesgue space} $L^p_\tau$ as the set of all functions $u\in L^p_\loc$ for whic the norm
\[\|u\|_{L^p_\tau}=\left(\int_M |u|^pr^{\tau p-n}\, dx\right)^{1/p}\]
is finite. 
For $k\in\Bbb N$, we define the \emph{weighted Sobolev space}  $W^{k,p}_\tau$ as the set of all functions $u\in W^{k,p}_\loc$ for which the norm 
\[\|u\|_{W^{k,p}_\tau}=\sum_{i=0}^k\|\nabla^i u\|_{L^p_{\tau+i}}\]
is finite. The space of compactly supported smooth functions, $C^\infty_c$, is dense in $W^{k,p}_\tau$.
\end{defn}

\begin{thm}[Strong stability]\label{strongstab}
Let $\varphi$ be a Lipschitz function on $\Sigma_\infty$ and suppose there exists a constant $\alpha\in\Bbb R$ such that 
\begin{equation}\label{condition} \varphi-\alpha \in \begin{cases} 
W^{1,2}_{\frac{n-3}{2}}(\Sigma_\infty) & \text{if $n\ge 4$} \\
C^1_c(\Sigma_\infty) & \text{if $n=3$}
\end{cases}.
\end{equation}
Then 
\begin{equation}\label{str}
    \int_{\Sigma_\infty}|\nabla\varphi|^2 +\frac 12 R_{\Sigma_\infty}\varphi^2 -\frac 12(R_M+h^2+2\nu(h))\varphi^2\,d\mathcal H^{n-1}\ge 0.
\end{equation}
\end{thm}
\begin{cor}\label{FINALLY}
Let $L_\gamma$ be the conformal Laplacian of the induced metric $\gamma$ on $\Sigma_\infty$. Then 
\[\int_{\Sigma_\infty}\frac{1}{2}(R_M+h^2-2|\nabla h|)\varphi^2\,d\mathcal H^{n-1}\le \int_{\Sigma_\infty}\varphi\,L_\gamma\varphi\,d\mathcal H^{n-1}\]
for any $\varphi$ satisfying \eqref{condition}.
\end{cor}

 First, we record the stability inequality satisfied by the approximating hypersurfaces $\Sigma_j=\Sigma_{\sigma_j}.$

\begin{prop}[Stability of the approximators] \label{approxstab}
Let $\alpha\in\Bbb R$, $R\ge R_0$,
and $\varphi$ be a function on $\Sigma_\infty$ 
that equals $\langle\nu_\infty,\alpha\partial_n\rangle$ for $|x|\ge R$.
Then there exists a vector field $X$ on $M$ with the following properties:
\begin{enumerate}
    \item $X =\alpha\partial_n$ for $|x|\ge 10R$,
    
    \item $\varphi =\langle X,\nu_\infty\rangle$ on $\Sigma_\infty$,
    
    \item $\delta^2_X\mathcal F_{\sigma_j}(\Sigma_j)\ge 0$ for all $j$ sufficiently large under the flow of $X$; see Proposition \ref{prop:2ndvar}, 
    
    \item If $b_\infty$ and $d_\infty$ denote the integrands \eqref{b} and \eqref{d} associated to $\Sigma_\infty$ and the flow of $X$, then 
   \begin{align}
       b_\infty&= |\nabla\varphi|^2 -(\Ric(\nu_\infty,\nu_\infty)+|A|^2)\varphi^2\\
       d_\infty&=-(h^2+\nu_\infty(h))\varphi^2
   \end{align}
   everywhere on $\Sigma_\infty$.
\end{enumerate}
\end{prop}
\begin{proof}
We seek a vector field $X$ on $M$ with the following properties:
\begin{enumerate}[(i)]
    \item $X=\alpha\partial_n$ for $|x|\ge 10 R$,
    
     \item $\varphi = \langle X,\nu_\infty\rangle$ everywhere on $\Sigma_\infty$,
    
    \item $\hat X=0$ on $\Sigma_\infty$ for $|x|\le R$, and
    
    \item $Z=0$ on $\Sigma_\infty$ for $|x|\le R$.
\end{enumerate}
Let $\theta_1$ be a smooth cutoff function $\equiv 1$ for $|x|\le R$ and $\equiv 0$ for $|x|\ge 10R$ with the property that there exists another smooth cutoff $\theta_2$ satisfying $\theta_1^2+\theta_2^2=1$ everywhere. Define a vector field $X_1=\theta_1\varphi \nu_\infty$ on $\Sigma_\infty$. Extend it by normal geodesics to a neighborhood of $\Sigma_\infty\cap \{|x|\le 10R\}$ and arbitrarily beyond it. Then set $\tilde X_1=\theta_1X_1$, globally. In the asymptotic region, we simply set $\tilde X_2=\alpha\theta_2^2\partial_n$, which is a globally defined vector field. It is easy to see that this vector field has the desired properties (i)-(iv). 

Now we explain how these imply (1)-(4). Clearly (1) follows from (i) and (2) from (ii). (3) follows from the height picking procedure and our choice of $a(\sigma_j)$, Lemma \ref{lem:heights}. Finally, (4) follows from plugging in (iii) and (iv) into equations \eqref{b} and \eqref{d} and using the fact that $h=0$ for $|x|\ge R\ge R_0$. 
\end{proof}


Taking $\alpha=0$ in Proposition \ref{approxstab} gives Theorem \ref{strongstab} for compactly supported $\varphi$, which is all we require when $n=3$.  The next step in the proof of Theorem \ref{strongstab} is a careful analysis of the limit $j\to\infty$ in 
$\delta^2_X\mathcal F_{\sigma_j}(\Sigma_j)\ge 0.$

\begin{lem}\label{adecay} For a fixed choice of $X$, the second variation integrands satisfy
\[a_j=O(|x'|^{-(n-1)})\] uniformly in $j$ (also allowing $j=\infty$). 
\end{lem}
It then follows from basic measure theory that $a_\infty\in L^1(\Sigma_\infty)$ and that
\begin{equation}\label{measuretheory}
  \int_{\Sigma_\infty} a_\infty \,d\mathcal H^{n-1}+\int_{\Sigma_\infty} d_\infty \,d\mathcal H^{n-1}\ge 0.
\end{equation}
For details in taking this limit, see \cite[4.4.2 {\emph{Stabilit\"at von $\Sigma$, Teil 1}}]{Kuwert}.
\begin{proof}[Proof of Lemma \ref{adecay}]
The proofs in Section \ref{sec:asymptotics} in fact show that $\Sigma_\sigma\cap \mathcal E'$ can be represented as a graph $u_\sigma$ as in Proposition \ref{AllardCorollary}, with a uniform $C^{1,\alpha}_0$ bound. 

Since $a_j$ is 2-homogeneous in $X$ ($Z$ scales quadratically), we may assume $\alpha=1$ for simplicity. In the asymptotic region, the flow of $X$ is simply vertical translation: $F(x,t)=x+te_n$ for $|x|\ge 10R$. Therefore, the acceleration vector field (with respect to $g$) is $\nabla_{\partial_n}^g\partial_n$. Using the transformation formula for the connection under a conformal change \cite[Theorem 1.159]{Besse},
\begin{align}\label{Z}Z&=\Gamma^\mu_{\overline g\,nn}\partial_\mu+\frac{2}{n-2}\left(2\partial_n\log\psi\partial_n-\overline g_{nn}\overline g^{\mu\nu}\partial_\mu\log\psi\partial_\nu\right)\\
&= C^{1,\alpha}_n \cdot\partial_\mu + C^{1,\alpha}_{n-1}\cdot\partial_\mu \nonumber\end{align}
From this it is apparent that $\Div_\Sigma^g Z = O(|x'|^{-n})$. (In fact, it appears to decay faster but this is all we shall require.)

Next, we analyze $\Div_{\Sigma_j}^g \partial_n=\Div^g\partial_n -\langle\nabla_{\nu_j} \partial_n,\nu_j\rangle$. Indeed,
\[\Div_g\partial_n=\frac{1}{\sqrt{\det g}}\partial_n\sqrt{\det g}=\Gamma^\mu_{g, n\mu}=O(|x'|^{-(n-1)}).\]
The term $\langle\nabla_{\nu_j} \partial_n,\nu_j\rangle$ is also on the level of a Christoffel symbol of $g$, hence $(\Div_{\Sigma_j}^g\partial_n)^2$ decays rapidly enough for our purposes. 

The terms involving $\langle e_i,\nabla_{e_j}X\rangle \langle \nabla_{e_i}X,e_j\rangle$ are on the the level of $\Gamma_{g}^2$, hence also decay rapidly enough. This same argument works for $|(\nabla_{e_i}X)^\perp|^2=\langle \nabla_{e_i}X,\nu_j\rangle^2$.

Finally, the Riemann tensor term is $O(|x'|^{-n})$ by Proposition \ref{table}. 
\end{proof}

\begin{lem}\label{cdecay} Along the domains $D_\sigma=\Sigma_\infty\cap Z_\sigma$, we have the limit
\[\lim_{\sigma\to\infty} \int_{\partial D_\sigma }c_\infty\, d\mathcal H^{n-2}=0.\]
\end{lem}
From now on we may drop the $\infty$ subscript when there is no risk for confusion since we are not working with the $\Sigma_j$'s any more. 

\begin{proof} Since $\mathcal H^{n-2}(\partial D_\sigma)= O(\sigma^{n-2})$, we merely require that $c=O(|x'|^{-(n-1)})$. For reference, the integrand is given by 
\[c=\langle Z,\eta\rangle + (\Div_\Sigma^g\partial_n)\langle \partial_n,\eta\rangle -\langle \nabla^g_{\partial_n}\partial_n,\eta\rangle -2\langle \partial_n,\nu\rangle A(\hat{\partial_n},\eta),\]
where $\hat{\partial_n}$ is the tangential component and $\eta$ is the conormal to $\partial D_\sigma$ in $\Sigma$. The first term has sufficient decay per \eqref{Z}. The second and third terms also decay to order $n-1$, as shown in the previous lemma. The last term has the desired decay rate by Lemma \ref{improved}.
\end{proof}

We can now prove the strong stability inequality.

\begin{proof}[Proof of Theorem \ref{strongstab}]
From the previous work, \eqref{str} holds for any $\varphi$ which equals $\langle \nu,\alpha\partial_n\rangle$ outside a compact set. Indeed, we already remarked that in this case, \eqref{measuretheory} holds. By measure theory,
\begin{align*}
    \int_{\Sigma_\infty}a_\infty\,d\mathcal H^{n-1}&=\lim_{\sigma\to\infty}\int_{D_\sigma} a_\infty\, d\mathcal H^{n-1}\\
    &=\lim_{\sigma\to\infty}\int_{D_\sigma}b_\infty \,d\mathcal H^{n-2}+\lim_{\sigma\to\infty}\int_{\partial D_\sigma}c_\infty \,d\mathcal H^{n-2}\\
    &=\int_{\Sigma_\infty}b_\infty \,d\mathcal H^{n-2}
\end{align*}
by the integration by parts formula \eqref{ibps} and Lemma \ref{cdecay}, as long as $b_\infty\in L^1(\Sigma_\infty)$. We now justify this. By Proposition \ref{table} and Lemma \ref{improved},
\[(\Ric(\nu,\nu)+|A|^2)\varphi^2=O(|x'|^{-n}),\]
hence is summable. For the upward pointing normal, we have
\[\langle \nu,\alpha\partial_n\rangle=\alpha\frac{\psi^{\frac{4}{n-2}}}{|d\zeta|_{\overline g}}.\]
It is easy to see that the derivatives of this function are $O(|x'|^{-(n-1)})$, so $|\nabla\varphi|^2\in L^1(\Sigma_\infty)$ as well. 

On $\Sigma_\infty$, the traced Gauss equation reads \cite[Corollary 2.7]{Lee} 
\[\Ric(\nu,\nu)+|A|^2 =\frac{1}{2}(R_M-R_{\Sigma_\infty}-|A|^2+h^2).\]
The scalar curvature of $M$ is integrable over $\Sigma_\infty$ by Proposition \ref{table} and the same calculation using Corollary \ref{SigmaAF} shows that $R_{\Sigma_\infty}$ is integrable as well. Putting this into the stability inequality shows 
\begin{equation}\label{str'}
    \int_{\Sigma_\infty}|\nabla\varphi|^2 +Q\varphi^2\,d\mathcal H^{n-1}\ge 0,
\end{equation}
where \[-2Q=R_M-R_{\Sigma_\infty}+|A|^2 +h^2+2\nu(h),\]
for any $\varphi$ which equals $\langle \nu,\alpha\partial_n\rangle$ outside of a compact set. 

Passing from such ``compactly supported" $\varphi-\langle\nu,\alpha\partial_n\rangle$ to the general case is now as in \cite[page 25]{EHLS}, by density.
\end{proof}

\subsection{Finishing the argument when $R_g>0$ far out}

We are finally ready to prove Theorem \ref{thm2}, with a modified assumption (2) that will be taken care of in the next section. 

\begin{prop}\label{basicallyfinal}
Let $(M^n,g)$ be an asymptotically Schwarzschild manifold. Let $U_1$ and $U_2$ be neighborhoods of infinity with $\overline{U_2}\subset U_1$ and $D>0$. If furthermore
\begin{enumerate}
    \item $g$ has no points of incompleteness in the $D$-neighborhood of $U_1$,
    
    \item $R_g\ge 0$ in the $D$-neighborhood of $U_1$ and $R_g>0$ for $|x|$ sufficiently large, and 
    
    \item the scalar curvature satisfies the largeness assumption 
    \[R_g> \sigma(\rho,D)\]
    on $\overline{U_1}\setminus U_2$,
\end{enumerate}
then $m\ge 0$. 
\end{prop}
\begin{proof}[Proof of Proposition \ref{basicallyfinal} when $n=3$]
We utilize the logarithmic cutoff trick (this step requires only quadratic area growth), i.e. plug
\[\varphi_\sigma(x')=\begin{cases} 
1 & \text{if $|x'|\le\sigma$} \\
2-\frac{\log |x'|}{\log\sigma} & \text{if $\sigma<x'<\sigma^2$}\\
0 & \text{if $\sigma^2<|x'|$}
\end{cases}\]
into the stability inequality \eqref{str}. The kinetic term is then 
\[\int_{\Sigma_\infty}|\nabla\varphi_\sigma|^2\,d\mathcal H^2\approx \int_\sigma^{\sigma^2} \frac{1}{r^2}\,r\,dr=\frac{2}{\log \sigma}\to 0,\] so we have
\[\int_{\Sigma_\infty} K \,d\mathcal H^2>0,\]
where $K$ is the Gauss curvature of $\Sigma_\infty$ and we get strict positivity because the strict $\mu$-bubble condition holds sufficiently far out by assumption (2). It is standard that the geodesic curvature of $\partial D_\sigma$ is $\sigma^{-1}+O(\sigma^{-2})$ \cite[page 255]{CMbook} and hence \[\int_{D_\sigma}k_g\,d\mathcal H^1=2\pi+O(\sigma^{-1}).\]
By the Gauss--Bonnet theorem
\[\int_{D_\sigma}K\,d\mathcal H^2=2\pi(\chi(D_\sigma)-1)+O(\sigma^{-1}).\]
Since we always have $\chi(D_\sigma)\le 1$ for a connected surface with boundary, this gives a contradiction. 
\end{proof}

\begin{proof}[Proof of Proposition \ref{basicallyfinal} when $4\le n\le 7$]
Here we follow \cite{EHLS}. We seek a conformal factor $w>0$, $w\to 1$ at $\infty$, so that $w^\frac{4}{n-3}\gamma$ is scalar flat. This is equivalent to solving $L_\gamma v=-R_{\Sigma_\infty}$ where $v=w-1$, with $v\to 0$ at $\infty$. The scalar curvature of $\Sigma_\infty$ is $O(|x'|^{-n})$ (which is integrable!) and lies in $L^p_{n-1-\ve}$ for any $\ve\in (0,1)$ and $p\ge 1$, which we can fix to be $>n$. We now claim that 
\[L_\gamma:W^{2,p}_{n-3-\ve }\to L^p_{n-1-\ve}\]
is an isomorphism. 

Suppose $\varphi\in W^{2,p}_{n-2-\ve}$ solves $L_\gamma w=0$. Then by elliptic regularity and Sobolev embedding, $\varphi\in C^1_{n-3-\ve}\cap C^\infty_\loc \subset W^{1,2}_{\frac{n-3}{2}}.$ Hence according to Lemma \ref{FINALLY} and the weak $\mu$-bubble condition satisfied by $h$,
\[\int_{\Sigma_\infty}\frac{1}{2}(R_M+h^2-2|\nabla h|)\varphi^2 =0.\]
Since $h$ satisfies the strict $\mu$-bubble condition for large $|x'|,$ $\varphi$ must vanish for large $|x'|$. Since $\varphi$ solves a homogeneous elliptic equation, strong unique continuatation \cite[Remark 6.2]{CMbook} implies it vanishes identically. Now by standard Fredholm-type theorems in asymptotic analysis \cite[Theorems A.40, A.42]{Lee}, $L_\gamma$ is an isomorphism. 

Let $v$ be the solution so obtained and set $w=v+1$. Using the fact that the conformal Laplacian has a positive principal eigenfunction with positive eigenvalue on each $D_\sigma$ for $\sigma$ sufficiently large, it is an easy matter to verify that $w>0$ via the maximum principle. By standard decay theorems (such as Theorem \ref{Meyers}), $w^\frac{4}{n-3}\gamma$ is asymptotically flat and has mass 
\[-\frac{2}{(n-3)|S^{n-2}|}\int_{\Sigma_\infty} w\,L_\gamma w\,\mathcal H^{n-1}.\]
By Lemma \ref{FINALLY}, this is strictly negative since $w\to 1$ at $\infty$. We have thus obtained a contradiction the positive mass theorem in dimension $n-1$. 
\end{proof}

\section{Proofs of the main theorems}\label{sec:thm1}

We formalize the observation made in Step 1 of \cite{SY79PMT}. 

\begin{lem}
Let $(M^n,g)$ be a asymptotically Schwarzschild manifold with negative mass and nonnegative scalar curvature on $\mathcal E$. For $r_0>0$ sufficiently large (depending on $g$), there exists a conformal factor $\varphi$ with the following properties:  
\begin{enumerate}
    \item The conformal metric $\tilde g=\varphi^\frac{4}{n-2}g$ is asymptotically Schwarzschild,
    
    \item $R_{\tilde g}>0$ for $|x|\ge 2r_0$, and
    
    \item $\varphi$ is constant equal to $1+O(r_0^{-(n-2)})$ for $|x|\le r_0$ and on $M\setminus\mathcal E$.
    
\end{enumerate}
\end{lem}
\begin{proof}
By \eqref{step1},
\[\Delta_g|x|^{-(n-2)}=m(n-2)^2 |x|^{-(2n-2)}+O(|x|^{-(2n-1)}),\]
which shows the existence of a $\sigma_0>0$ such that $-\Delta_g|x|^{-(n-2)}>0$ for any $|x|\ge \sigma_0$. Let $t_0=-m/8r_0^{n-2}>0$ for $r_0\ge \sigma_0$. 

Let $\zeta$ be a smooth function on $(0,\infty)$ with the following properties:
\[\zeta(t)=\begin{cases}
t & \text{for $t<t_0$}\\
\tfrac 32 t_0 & \text{for $t>2t_0$}
\end{cases},\quad \zeta'\ge 0,\quad\text{and}\quad \zeta''\le 0.\]
On $\mathcal E$, we define $\varphi$ by
\[\varphi(x)=1+\zeta\left(-\frac{m}{4|x|^{n-2}}\right).\]
If $|x|\le r_0$, then $\varphi(x)=1+\frac 32 t_0=1+O(r_0^{-(n-2)})$ and we extend it by this constant to all of $M$. If $|x|\ge 2r_0$, then 
\[\varphi(x)=1-\frac{m}{4|x|^{n-2}}.\]
From this is is apparent that $\tilde g$ is asymptotically Schwarzschild with mass $\frac{m}{2}<0$. Finally, the scalar curvature is given by the usual formula \[R_{\tilde g}= \varphi^{-\frac{n+2}{n-2}}\left(-4\frac{n-1}{n-2}\Delta_g \varphi+R_g\varphi\right).\]
The Laplacian term has the correct sign, as is easily checked:
\[-\Delta_g\varphi(x) = \frac{m}{4}\zeta'' |\nabla |x|^{-(n-2)}|^2+\frac{m}{4}\zeta'\Delta_g|x|^{-(n-2)}\ge 0\]
everwhere, and $>0$ for $|x|\ge 2t_0$. 
\end{proof}

This allows us to remove the strict positivity assumption in Proposition \ref{basicallyfinal}.

\begin{proof}[Proof of Theorem \ref{thm2}]
We apply the conformal transformation of Lemma \ref{step1} to $(M,g)$ to bring it into the form of Proposition \ref{basicallyfinal}. Indeed, since $\overline{U_1}\setminus U_2$ is compact and $R_g$ is continuous, $R_g-\sigma(\rho,D)$ has a positive minimum there. It follows that if we choose $r_0$ sufficiently large in Lemma \ref{step1}, then $R_{\tilde g}>\sigma(\rho,D)$ on $\overline{U_1}\setminus U_2$ as well. Furthermore, distances will be distorted an arbitrarily small amount. So if the original mass is negative, we have a counterexample to Proposition \ref{basicallyfinal}. 
\end{proof}

\begin{proof}[Proof of Theorem \ref{thm1}]
If $(M,g)$ is complete and has nonnegative scalar curvature everywhere, apply Lemma \ref{step1} with some $r_0$ to make the scalar curvature strictly positive on the end. We may then take $U_i$ as in Remark \ref{rkbig} (4) sufficiently far out that the scalar curvature is strictly positive on the compact set $\overline{U_1}\setminus U_2$. Now $\rho$ is fixed and since for fixed $\rho$,
\[\lim_{D\to\infty}\sigma(\rho,D)=0,\]
we can find $D$ sufficiently large that $R_{\tilde g}>\sigma(D,\rho)$ on $\overline{U_1}\setminus U_2$. But $(M,g)$ is complete with nonnegative scalar curvature globally, so we satisfy all the hypotheses of Theorem \ref{thm2} with this choice of $\rho$ and $D$. It follows that the original mass is nonnegative.
\end{proof}

\begin{appendix}

\section{Curvature estimate in homogeneously regular manifolds}\label{app:A}

In this appendix we give an independent proof Lemma \ref{curvatureestimate}. 

\begin{defn}
Let $K\ge 0$ and $i_0>0$ be real numbers. We say that a complete Riemannian manifold $(M^n,g)$ is \emph{$(K,i_0)$-homogeneously regular} if its sectional curvatures are bounded in absolute value by $K$ and injectivity radius bounded below by $i_0$. In the case of a noncompact manifold, we use the same definition, with the understanding that we work at a distance at least $2i_0$ away from any incomplete points; see Definition \ref{incomplete}.
\end{defn}

In the rest of this section, $K$ and $i_0$ are fixed and all manifolds are assumed to be $(K,i_0)$-homogeneously regular. We work with geodesic balls $B_\rho(x_0)^g$ and use the phrase ``$\Sigma$ is a minimal/minimizing submanifold of $B_\rho^g(x_0)$" to mean that $\Sigma\subset B_\rho^g(x_0)$ and that $\partial\Sigma\subset\partial B_\rho^g(x_0)$. 

\begin{thm}\label{rescaledcurvatureestimate}
Let $2\le n\le 7$, $K\ge 0$, and $i_0>0$. There exist universal constants $\rho_0\in(0,i_0)$ and $C$ depending only on $n, K$, and $i_0$ with the following property. Let $(M^{n},g)$ be a $(K,i_0)$-homogeneously regular Riemannian manifold, $x_0\in M$, and $\rho\in (0,\rho_0)$. If $\Sigma$ is an area minimizing hypersurface in $B_\rho^g(x_0)$ with $x_0\in \Sigma$, then 
\[|A_\Sigma(x_0)|^2\le \frac{C}{\rho^2}.\]
\end{thm}

In our definition of homogeneously regular, we can only obtain coordinates in which the metric is $C^{1,\alpha}$-controlled, see Lemma \ref{HH} below. Therefore, it might seem like this is a regularity theorem for $C^{1,\alpha}$ parametric elliptic functionals, but G.\ De Philippis has pointed out that this is still essentially perturbative about the flat space regularity theory (due to the freedom of choosing $\rho_0$ very small). This can be formalized using the work of Tamanini \cite{Tamanini}. In this appendix, we give an alternative proof via a contradiction argument, in the spirit of White \cite{White05} (for an Allard-type result) and Simon \cite{SimonRemark} (for a universal density bound).

\begin{thm}[Allard-type theorem] \label{allard}
Let $n\in\Bbb N$, $K\ge 0$, and $i_0>0$. There exist universal constants $\rho_0\in(0,i_0)$, $\ve_0>0$, $C$, and $\tau>0$ depending only on $n, K$, and $i_0$ with the following property. Let $(M^{n},g)$ be a $(K,i_0)$-homogeneously regular Riemannian manifold, $x_0\in M$, $\rho\in (0,\rho_0)$, and $\Sigma$ a minimal hypersurface in $B_\rho^g(x_0)$ with $x_0\in \Sigma$. If
\[\Theta(x_0,\rho)\le 1+\ve_0,\]
then 
\[\sup_{\Sigma\cap B_{\tau\rho}^g(x_0)}|A|^2\le \frac{C}{\rho^2}.\]
\end{thm}

Note that this does not follow directly from Allard's theorem \cite{Allard} because it is not possible to bound the mean curvature of $\Sigma$ with respect to the underlying Euclidean structure.

\begin{prop}[Universal density bound]
\label{densitybound}
Let $2\le n\le 7$, $K\ge 0$, and $i_0>0$. There exists a universal constant $\delta_0\in(0,1)$ depending only on $n, K$, and $i_0$ with the following property. Let $(M^{n},g)$ be a $(K,i_0)$-homogeneously regular Riemannian manifold, $x_0\in M$, and $\rho\in (0,\rho_0)$. If $\Sigma$ is an area minimizing hypersurface in $B_\rho^g(x_0)$ with $x_0\in \Sigma$, then 
\[\Theta(x_0,\delta_0\rho)\le 1+\ve_0,\]
where $\ve_0$ is as in Theorem \ref{allard}. 
\end{prop}

Combining Theorems \ref{allard} and \ref{densitybound} immediately gives Theorem \ref{rescaledcurvatureestimate}. To prove these theorems, we recall a basic result about existence of harmonic coordinate charts. 

\begin{lem}[Hebey--Herzlich \cite{HH}]\label{HH}
Let $(M^n,g)$ be $(K,i_0)$-homogeneously regular and  $\alpha\in(0,1)$. There exists a constant $Q_0$ and a number $r_0$ (depending only on $K, i_0$, $n$, and $\alpha$) such that for any $x_0\in M$, there exists a harmonic coordinate chart $x^i$ on $B_{r_0}^g(x_0)$ sending $x_0$ to the origin and satisfying 
\begin{equation}
\label{quadform}
Q_0^{-1}\delta_{ij}\le g_{ij}\le Q_0g_{ij}\quad\text{as quadratic forms},\end{equation}
and
\begin{equation}
\label{derivativeHolder}
\|\partial_k g_{ij}\|_{C^{0,\alpha}(U)}\le Q_0,\end{equation}
where $U\subset \Bbb R^n$ is the domain of definition of the harmonic coordinates. 
\end{lem}

For a metric $g$, we define the area ratio at a point $x_0\in M$ and scale $r$ according to 
\[\Theta_\Sigma^g(x_0,r)=\frac{\mc H^{n-1}_g(\Sigma\cap B_r^g(x_0))}{\omega_n r^n}.\]
The following is easily proved using the Hessian comparison theorem \cite[page 234]{CMbook}.

\begin{lem}[Monotonicity formula]
On any minimal hypersurface in a $(K,i_0)$-homogeneously regular manifold,
\begin{equation}
    \frac{d}{dr}\left(e^{(n-1)\sqrt{K} r}\Theta(x_0,r)\right) \ge 0
\end{equation}
for $r< i_0$.
\end{lem}

Using this, we can bound the area ratio at nearby points; compare \cite[Lemma D.7]{Ecker}.

\begin{lem}\label{moving} Assume the hypotheses of Theorem \ref{allard} and also let $\ve>0$. Then there exist constants $\beta, \eta>0$ and $\rho_0\in (0,\min\{r_0,i_0\})$ (depending only on $n,K,i_0$, and $\ve$ and where $r_0$ is the harmonic radius from Lemma \ref{HH} with an otherwise arbitrary choice of $\alpha$) with the property that  if
\[\Theta(x_0,\rho)\le 1+\ve\] 
for some $\rho\in (0,\rho_0)$, then
\[1-\eta\le \Theta(y,\sigma)\le 1+2\ve\]
for every $y\in B^g_{\beta\rho}(x_0)\cap \Sigma$ and $\sigma<(1-\beta)\rho$. With all other parameters fixed, $\lim_{\ve\to 0}\beta=\lim_{\ve\to 0}\eta=\lim_{\ve\to 0}\rho_0=0$. 
\end{lem}
\begin{proof}
Using the monotonicity at $y\in B^g_{\beta}(x_0)$ as well as the inclusion $B_{(1-\beta)\rho}(y)\subset B_\rho(x_0)$ we estimate
\begin{align*}
  \Theta(y,\sigma)&\le e^{(n-1)\sqrt K((1-\beta) \rho-\sigma )} \Theta(y,(1-\beta)\rho)\\
    &\le  \frac{e^{(n-1)\sqrt K((1-\beta)\rho -\sigma )}}{(1-\beta)^n}(1+\ve) \\
    &\le 1+2\ve
\end{align*}
for $\beta$ and $\rho_0$ well chosen.

For the lower bound, we note that monotonicity implies 
\[\Theta(y,\sigma)\ge  e^{-(n-1)\sqrt K\sigma}.\]
If $\sigma$ is sufficiently small, then we can write this as 
\[\Theta(y,r)\ge 1-\eta(\ve).\qedhere\]
\end{proof}

\begin{proof}[Proof of Theorem \ref{allard}]
By Lemma \ref{moving} it suffices to show that there exist constants $\ve_0$ and $C$ such that 
\[\sup_{y\in B_{\beta\rho}^g(x_0)\cap \Sigma,\sigma <(1-\beta)\rho}\Theta_\Sigma^g(y,\sigma)\le 1+\ve_0\] implies
\[\sup_{\sigma<\beta\rho}\left(\sigma^2\sup_{B_{\beta\rho-\sigma}(x_0)}|A_\Sigma|^2\right)\le C,\]
where $\rho<\rho_0(\ve_0)$. Indeed, if this is true we may take $\tau=\frac \beta 2$ and $\sigma=\tau \rho$.

Arguing by contradiction, we infer the existence of $(K,i_0)$-homogeneously regular manifolds $(M^n_j,g_j,x_j)$, numbers $\ve_j\to 0$, minimal hypersurfaces $\Sigma_j\subset B^{g_j}_{\rho_j}(x_j)$, $\rho_j\in (0,\rho_0(\ve_j))$, such that 
\begin{equation}\label{z}\sup_{y\in B^{g_i}_{\beta_j\rho_j}(x_j)\cap \Sigma_j,\sigma<(1-\beta_j)\rho_j}\Theta_j(y,\sigma)\le 1+\ve_j,\end{equation}
but
\[\gamma_j^2:= \sup_{\sigma<\beta_j\rho_j}\left(\sigma^2\sup_{B_{\beta_j\rho_j-\sigma}(x_j)\cap\Sigma_j}\right)\to\infty.\]
We find $\sigma_j<\beta_j\rho_j$ and $y_j\in \overline B_{\beta_j\rho_j-\sigma_j}(x_j)$ such that 
\[\gamma_j^2 =\sigma_j^2\sup_{ B_{\beta_j\rho_j-\sigma_j}(x_j)\cap\Sigma_j}|A_j|^2=\sigma_j^2|A_j(y_j)|^2.\]
From the definitions, we see
\[\sup_{B_{\beta_j\rho_j-\frac{\sigma_j}{2}}(x_j)\cap \Sigma_j}|A_j|^2\le 4|A_j(y_j)|^2\]
and by the inclusion $B_{\frac{\sigma_j}{2}}(y_j)\subset B_{\beta_j\rho_j-\frac{\sigma_j}{2}}(x_j)$,
\[\sup_{B_\frac{\sigma_j}{2}(y_j)}|A_j|^2 \le 4|A_j(y_j)|^2.\]

By Lemma \ref{HH} and our choice of $\rho_0$, we can introduce harmonic coordinates $y_j^i$ on $B_{\frac{\sigma_j}{2}}^{g_j}(y_j)$ (which we can now view as a subset of $\Bbb R^n$) sending $y_j$ to $0$ and satisfying \eqref{quadform} and \eqref{derivativeHolder}.  Let $\lambda_j=|A_j(y_j)|^{-1}=o(j)$ and rescale $(B^{g_j}_{\frac{\sigma_j}{2}}(y_j),g_j)$ by $\lambda_j$. Since $\frac{\sigma_j}{2}\lambda_j^{-1}\to \infty,$ $\bigcup_{j}B^{g_j}_{\frac{\sigma_j}{2}}(y_j)=\Bbb R^n$. Furthermore, the components of the rescaled metrics are given by $(g_j)_{kl}(\frac{y}{\lambda_j})$, so by \eqref{quadform} and \eqref{derivativeHolder} (sub-)converge in $C^{1,\alpha}_\loc(\Bbb R^n)$ to the Euclidean metric.  By essentially the same calculation leading to Lemma \ref{comparison}, Proposition \ref{bigeqn}, and Schauder theory \cite[Chapter 6]{GT}, the rescaled hypersurfaces $\lambda_j^{-1}\Sigma_j$ (sub-)converge in $C^{2,\alpha}_\loc$ to a complete, smooth minimal hypersurface in $\Bbb R^n$. (The coefficients of the minimal graph-type equations satisfied by the $\lambda_j^{-1}\Sigma_j$'s converge in $C^{0,\alpha}_\loc$ to their Euclidean values, which suffices for this application of Schauder theory.) The limiting hypersurface $\Sigma$ has $|A(0)|=1$. 

Finally, we need to show that $\Theta_\Sigma(0,R)=1$ for every $R$. This follows from the upper bound \eqref{z} and the lower bound in Lemma \ref{moving}. Therefore $\Sigma$ is a cone, but as it is smooth, it must be a hyperplane. This contradicts $|A(0)|=1$.
\end{proof}

\begin{proof}[Proof of Proposition \ref{densitybound}]
We argue again by contradiction. There exist Riemannian manifolds $(M^n_j,g_j,x_j)$, numbers $\delta_j\to 0$, and area minimizing hypersurfaces $\Sigma_j\subset B_{\rho_j}^{g_j}(x_j)$ with $x_j\in \Sigma_j$ such that 
\[\Theta_j(x_j,\delta_j\rho_j)\ge 1+\ve_0.\]
By the monotonicity formula,
\[\Theta_j(x_j,\rho)\ge e^{-(n-1)\sqrt K(1-\delta_j)\rho_j}(1+\ve_0)\ge 1+\frac{\ve_0}{2}\]
for any $\rho\in (\delta_j\rho_j,\rho_j)$, after possibly adjusting $\rho_0(\ve_0)$. We now rescale with scale factor $\lambda_j=(\delta_j\rho_j)^{-1/2}=o(j)$. By the same argument as in the proof of Theorem \ref{allard}, the rescaled metrics will (sub-)converge (in harmonic coordinates) to the Euclidean metric. By a slight modification of the proof of \cite[Theorem 34.5]{Simon} (using the fact that the area functionals with respect to the rescaled metrics are converging to the flat one in the obvious way), one can show that $\lambda_j^{-1}\Sigma_j$ converge to some (Euclidean) area minimizing current $T$ in the weak topology and as varifolds. Furthermore, $T$ will be multiplicity one and is carried by a smooth hypersurface $\Sigma$ (here we use $n\le 7$). By varifold convergence, we have 
\[\Theta_\Sigma(0,R)\ge 1+\frac{\ve_0}{2}\]
for every $R>0$. But $\Sigma$ is a smooth hypersurface, so its $(n-1)$-density at the origin is 1. This is a contradiction. 
\end{proof}

\end{appendix}

\bibliographystyle{acm}
\bibliography{PMTincomplete}

\end{document}